\documentclass{article}

\usepackage{geometry}
\usepackage{hyperref}
\usepackage{amsmath,amssymb,amsthm}
\usepackage{mathtools}
\usepackage{xcolor}

\newtheorem{thm}{Theorem}[section]
\newtheorem{lem}[thm]{Lemma}
\newtheorem{prop}[thm]{Proposition}
\newtheorem{cor}[thm]{Corollary}
\theoremstyle{remark}
\newtheorem{rem}[thm]{Remark}

\theoremstyle{definition}
\newtheorem{defn}[thm]{Definition}
\newtheoremstyle{Claim}{}{}{\itshape}{}{\itshape\bfseries}{:}{ }{#1}
\theoremstyle{Claim}

\newcommand{\T}{{\mathbb{T}^N}}

\renewcommand{\H}{\mathcal{H}}

\newcommand{\Rset}{\mathbb{R}}
\newcommand{\RsetN}{\mathbb{R}^N}

\newcommand\Kcal{{\mathcal{K}}}
\newcommand\Ecal{{\mathcal{E}}}

\DeclareMathOperator{\diverg}{div}

\titlepage
\title{Time-dependent focusing Mean-Field Games: \\ the sub-critical case}
\author{Marco Cirant and Daniela Tonon}
\date{\today}

\begin{document}

\maketitle

\begin{abstract} We consider time-dependent viscous Mean-Field Games systems in the case of local, decreasing and unbounded coupling. These systems arise in mean-field game theory, and describe Nash equilibria of games with a large number of agents aiming at aggregation. We prove the existence of weak solutions that are minimisers of an associated non-convex functional, by rephrasing the problem in a convex framework. Under additional assumptions involving the growth at infinity of the coupling, the Hamiltonian, and the space dimension, we show that such minimisers are indeed classical solutions by a blow-up argument and additional Sobolev regularity for the Fokker-Planck equation. We exhibit an example of non-uniqueness of solutions. Finally, by means of a contraction principle, we observe that classical solutions exist just by local regularity of the coupling if the time horizon is short.
\end{abstract}

\noindent
{\footnotesize \textbf{AMS-Subject Classification}}. {\footnotesize 35K55, 49N70.}

\section{Introduction}

Mean Field Games (MFG) theory models the behavior of an infinite number of indistinguishable rational agents aiming at minimising a common cost. The theory  was introduced in the seminal papers by Lasry and Lions \cite{LL061, LL062, LL07, Lcol} and by Huang, Caines and Malham\'e \cite{HCM06} to describe Nash equilibria in differential games with infinitely many players. A large part of  MFG literature is devoted to the study of MFG systems with increasing coupling. Heuristically, this assumption means that agents prefer sparsely populated areas (indeed concentration costs), and it is well-suited to model competitive cases. The increasing monotonicity of the coupling ensures existence and regularity of solutions in many circumstances (see, e.g., \cite{GoBook} and references therein); it is also a key assumption if one looks for uniqueness of equilibria. Although those models have a wide range of applications, they  rule out the possibility to apply the MFG theory to analyse aggregation phenomena, that is when agents aim at converging to a common state.  To cite an example, in \cite{Gue09}, Gu\'eant considered simple population models where  individuals have preferences about resembling to each other. Very few results exist in this direction and they only deal with very particular cases. See  \cite{Go16, gueant12} and \cite{BarPri}, for the quadratic and linear-quadratic case.
Our goal is to better understand this class of ``focusing'' MFG systems, where the coupling is monotone decreasing and it is a local function of the distribution, so that no regularising effect can be expected. Actually, non-existence, non-uniqueness of solutions, non-smoothness, and concentration are likely to arise, as shown by the first author in \cite{cir16}, where the stationary focusing case is considered.  In that paper, it is proven that there exists a threshold for the growth of the coupling, after which solutions to the MFG system may not even exist. Indeed, the focusing character of the MFG induce solutions to concentrate and develop singularities.

Let us enter into  the details of the  kind of MFG systems we deal with in this paper. In order to avoid boundary issues and exploit the compactness of the state space, we set our problem on the $N$-dimensional flat torus $\T$. Let $Q = Q_T = \T \times (0, T)$.  We consider MFG systems of the form
\begin{equation}\label{MFG}
\begin{cases}
-u_t - \Delta u + H(\nabla u) = -f(x,m(x,t)),& \text{in $Q$,} \\
m_t - \Delta m - \diverg(\nabla H(\nabla u)\, m)  = 0 & \text{in $Q$,} \\
m(x,0) = m_0(x), \quad u(x,T) = u_T(x) & \text{on $\T$}.
\end{cases}
\end{equation}
where $\int_\T m_0 \, dx =1$, $m_0>0$, $m_0, u_T \in C^2(\T)$.  In the system above, the first is the Hamilton-Jacobi-Bellman equation for the value function $u$ of a single agent, the second is the Kolmogorov-Fokker-Planck equation that governs the evolution of the distribution of the population $m$.

 Even more than in the competitive case, the assumptions on the Hamiltonian $H$, the growth of the coupling $f$ and the dimension of the state space will affect the qualitative behavior of the system. Let us clearly state the assumptions we make throughout the article.

The Hamiltonian $H : \RsetN \to \Rset$ is convex, $C^3(\RsetN)$, and has a superlinear growth: there exist $\gamma > 1$, $C_H > 0$ such that
\begin{equation}\label{Hass}
 C_H^{-1}|p|^\gamma \le H(p) \le C_H(|p|^{\gamma} + 1), \\
\end{equation}
for all $p \in \RsetN$. Its Legendre transform, $L(q) := \sup_{p \in \RsetN} \{p \cdot q - H(p)\}$ satisfies for some $C_L > 0$,
\begin{equation}\label{Lass}
 C_L^{-1}|q|^{\gamma'} - C_L \le L(q) \le C_L(|q|^{\gamma'} + 1) \\
\end{equation}
for all $q \in \RsetN$, where $\gamma'$ is the conjugate exponent of $\gamma$, i.e. $\frac 1 {\gamma}+\frac 1 {\gamma'} =1$.

The local coupling $f:\T\times[0,+\infty)\to\Rset$, $f \ge 0$ is continuous in both variables, differentiable w.r.t. the second variable  and satisfies, for $\alpha>0$,
\begin{equation}\label{fass}
 |\partial_m f(x,m)| \le c_f(m + 1)^{\alpha-1}
\end{equation}
for all $(x, m)\in \T\times[0,+\infty)$. Note that \eqref{fass} implies
\[
0 \le f(x, m) \le \frac{c_f}{\alpha} (m+1)^\alpha - \frac{c_f}{\alpha} + f(0) \le C_f (m^\alpha + 1),
\]
for all $(x, m)\in \T\times[0,+\infty)$, but in general $f$ is not bounded by above. Actually, what is important here is to have a control of the behaviour of $f$ at infinity, rather than requiring a restriction on the monotonicity of $f$. Note that $f$ could also depend explicitly on time, without giving any additional difficulty.
Let
\[
F(x,m) = \int_{0}^m f(x, \sigma) d \sigma \quad \forall (x, m)\in \T\times[0,+\infty),\quad F(x,m) =+ \infty \quad \text{otherwise,} 
\]
we then have, for all $m \ge 0$ and a $ C_F>0$,
\begin{equation}\label{Fass}
0 \le F(x, m) \le  C_F (m^{\alpha+1} + 1).
\end{equation}
  

Before stating our results, let us have a  look at the ones obtained in \cite{cir16}  for the stationary case.  In this setting, scaling properties of the system and regularity of the Kolmogorov equation can be exploited to prove that:
 if $\alpha < {\gamma'}/{N}$, there exists a classical solution basically by means of a control of the ``energy'' associated to the system. In other words, the decaying of the coupling is well compensated by the regularising properties of the diffusion.
 If ${\gamma'}/{N}\leq \alpha< {\gamma'}/{(N-\gamma')}$, such a control turns out to be more delicate, as aggregation may become the leading effect. Therefore, existence of classical solutions can be obtained only under additional assumptions on the coupling. In both cases, existence is shown by a blow-up method, together with Schauder's fixed point theorem. If $\alpha>  {\gamma'}/{(N-\gamma')}$, it is shown that classical solutions may not exist, i.e. concentration due to the fast decay of the coupling cannot be compensated by the diffusion.

In the evolutionary case, we expect a similar behaviour. 
Indeed, one of the tools that turns out to be fundamental for understanding a MFG system of the form \eqref{MFG} is its variational formulation: let $\Kcal \subset L^{\alpha+1}(Q) \times L^1(Q)$ be the pairs $(m,w)$ satisfying $|w|^{\gamma'}m^{1-\gamma'} \in L^1(Q)$, $m \ge 0$ and
\begin{equation}\label{kcalconstraint}
\int_Q (m \varphi_t + w\cdot \nabla \varphi - \nabla m \cdot \nabla \varphi ) dx dt+  \int_{\T}m_0(x) \varphi(x,0) dx = 0\quad \forall \varphi \in C_0^\infty(\T \times [0,T)),
\end{equation}
then the energy functional  $\Ecal$ associated with \eqref{MFG} is defined on $\Kcal$ as\footnote{The term $mL(-w/m)$ has to be intended equal to zero if $(m,w) = (0,0)$ and $+\infty$ if $w \neq 0$ and $m \le 0$.}
\begin{equation}\label{defE}
\Ecal(m, w) := \int_Q m L\left(-\frac{w}{m}\right) - F(x,m) \, dx dt + \int_\T u_T(x) m(x,T) \, dx.
\end{equation}
This functional shares many features with its stationary analogue, associated with the stationary problems considered in \cite{cir16}. On the other hand, some useful scaling properties are missing in the parabolic case by the presence of time.
Note that, due to the behaviour of  $f$ at infinity, $\Ecal$ is not convex and it may not be bounded from below in general. Hence, the usual variational methods that link minimisers of this functional to solutions of the MFG system cannot be used in general for this case. However, if we restrict to the regime $\alpha < {\gamma'}/{N}$, the energy
  $\mathcal {E}$ becomes bounded  from below; this key fact is proven in Lemma \ref{Eboundbelow}. Let us comment upon this last  assumption.
  In the stationary case,  using the rescaling properties of the Kolmogorov equation (that acts as a constraint), it can be shown that  it is a necessary and sufficient condition for the stationary version of the energy functional $\mathcal {E}$ to be bounded from below. In particular, if $\alpha \ge {\gamma'}/{N}$, by scaling a test competitor one observes that the term $-\int F(x,m)$ prevails on $\int m L(-{w}/{m})$, making the energy unbounded. In the evolutionary case, however, such a procedure does not apply directly, because any rescaling in the space variable changes the initial datum and jeopardises the constraint $\Kcal$. Nevertheless, this hypothesis is crucial in Proposition \ref{GN2},  where we prove estimates for a superlinear power of the  term $$\int_Q m^{\alpha+1} \, dxdt .$$ in terms of $\int m L(-{w}/{m})$. Such an estimate is based on the Gagliardo-Nirenberg inequality and elliptic regularity applied to the Fokker-Planck equation.

In  what follows we will always suppose
\[
\alpha < \frac{\gamma'}{N},
\]
so that $\Ecal$ is bounded by below on $\Kcal$: for this reason, we will call this regime {\it sub-critical}. Under this assumption, it can be shown through a convexification procedure that the energy functional $\mathcal {E}$ possess a minimum that can be linked to a weak solution (in the sense of Definition \ref{def:weaksolMFG}) of the MFG system. The idea is that variational techniques similar to the ones presented by the second author et alt. in \cite{CGPT}, can be applied to prove the existence of a weak solution to the convexified MFG system. Then it is enough to prove that the solution of the convexified problem provides a solution for the original one.  

When $\gamma'  > N+2$, thanks to Corollary \ref{comp}, more regularity can be proven for weak solutions, that are indeed classical ones.  

When $2 <\gamma' \le N+2$, we are able to prove the existence of classical solutions only under the additional hypothesis 
 \[
\alpha < \min\left \{ \frac{\gamma'}{N}, \frac{\gamma'-2}{N+2-\gamma'}\right\}.
\]
In this case a penalisation argument allows to obtain a sequence of solutions applying Theorem \ref{ex_clas}. Then a series of a-priori estimates will show the convergence of the sequence to a classical solution of the original problem. These estimates  rely on some blow-up techniques for which the hypothesis on $\alpha$ and the requirement $\gamma'>2$ (i.e. $H$ is subquadratic) are necessary.

The main theorems are then the following

\begin{thm}\label{ex_clas} Suppose that \eqref{Hass}, \eqref{fass} hold, and $\gamma'  > N+2$. Then, there exists a classical solution $(u, m)$ of \eqref{MFG} such that $(m , -m\nabla H(\nabla u))$ is a minimiser of $\Ecal$.
\end{thm}

\begin{thm}\label{ex_weak} Suppose that \eqref{Hass}, \eqref{fass} hold, and $1< \gamma'  \le N+2$. Then, there exists a weak solution $(u, m)$ of \eqref{MFG} such that $(m , -m\nabla H(\nabla u))$ is a minimiser of $\Ecal$.
\end{thm}

\begin{thm}\label{ex_clas_2} Suppose that \eqref{Hass}, \eqref{fass} hold, $2 <\gamma' \le N+2$ and \[
\alpha < \min\left \{ \frac{\gamma'}{N}, \frac{\gamma'-2}{N+2-\gamma'}\right\}.
\]
Then, there exists a classical solution $(u, m)$ of \eqref{MFG} such that $(m , -m\nabla H(\nabla u))$ is a minimiser of $\Ecal$.
\end{thm}

A convexification argument has already been used by  Briani and Cardaliaguet in \cite{BriCar} to prove the existence of classical solutions for the case of energy functionals where the coupling is a regularising function of $m$. In their setting, the energy not necessarily convex, but it is automatically bounded by below. Their result is part of a more general analysis on stability of solutions in MFG having multiple equilibria. 

Up to our knowledge, the only examples of non-uniquess of solutions available in the literature for the evolutionary problem are the ones presented in Briani and Cardaliaguet \cite{BriCar} and Bardi-Fisher \cite{BarFis}. The one we present here consists in a class of MFG systems for which the solution that minimises $\mathcal E $ is not equal to the trivial solution $(\bar u, \bar m)= ((t-T)f(1),1)$.  Indeed, starting from a solution of the associated stationary problem, it is possible to construct  a suitable competitor $(u,m)$ for which $\mathcal E (u,m)<\mathcal E (\bar u, \bar m)$.

Finally, one should expect similar features for the stationary and non-stationary problems when the time horizon $T$ is large. Conversely, if $T$ is small, the two settings may exhibit different features.
As mentioned before, existence of solutions might be false in the stationary case when the coupling is very strong. On the other hand, in the parabolic case a standard contraction argument applies: we prove the existence of a classical solution of \eqref{MFG} for small $T$, without requiring {\it any} assumption on the growth at infinity of $f$. 

\begin{thm}\label{short_ex} Suppose that $f, H \in C^3$. Then, there exists $T^* > 0$ such that for all $T \in (0,T^*]$, \eqref{MFG} has a classical solution.
\end{thm}

The idea is to exploit the local regularity of $f$ and $H$; even if the coupling is very strong, during a small time interval the distribution $m$ should remain close to the initial datum $m_0$, without developing singularities. 

We mention that the contraction theorem has already been used by Ambrose in \cite{Ambrose} in the MFG setting, in order to prove the existence of small, locally unique, strong solutions over any finite time interval in the case of a local coupling and a superquadratic Hamiltonian, when $m_0$ is chosen sufficiently close to a uniform distribution.

The paper is organized as follows:  we first collect embedding theorems and estimates for the Hamilton-Jacobi equation and  Fokker-Planck equation.  Section \ref{existence} is devoted to the proof of Theorems \ref{ex_clas}, \ref{ex_weak}  and \ref{ex_clas_2}. A non-uniqueness example is shown in Section \ref{non-uniqueness}. Finally, the appendix contains the proof of the existence of classical solutions for small $T$.
 
\bigskip

{\bf Acknowledgements.} The first author is partially supported by the Fondazione CaRiPaRo Project ``Nonlinear Partial Differential Equations: Asymptotic Problems and Mean-Field Games'' and the INdAM-GNAMPA project ``Fenomeni di segregazione in sistemi stazionari di tipo Mean Field Games a pi\`u popolazioni''. The second author is partially supported by  the  ANR project  MFG ANR-16-CE40-0015-01, the PEPS-INSMI Jeunes project ``SOME OPEN PROBLEMS IN MEAN FIELD
GAMES'' for the years 2016 and 2017,  and the PGMO project VarPDEMFG.

\section{Notations and preliminaries} 

We begin with the definition of some particular Banach  spaces involving time and space weak derivatives. Let $Q = Q_T = \T \times (0, T)$ and $r>1$. We denote by $W^{2,1}_r(Q)$ the space of functions $u \in L^r(Q)$ having weak space derivatives $D^\beta_x u \in L^r(Q)$ for $|\beta| \le 2$ and weak time derivative $ u_t \in L^r(Q)$, equipped with the norm
\[
\|u\|_{W^{2,1}_r(Q)} = \|u\|_{L^r(Q)} + \| u_t\|_{L^r(Q)} + \sum_{1 \le |\beta| \le 2} \|D^\beta_x u\|_{L^r(Q)}.
\]
Similarly, $W^{1,0}_r(Q)$ is endowed with the norm
\[
\|u\|_{W^{1,0}_r(Q)} = \|u\|_{L^r(Q)} + \sum_{|\beta| = 1} \|D^\beta_x u\|_{L^r(Q)}.
\]
Finally, we denote by $\H^{r,1}(Q)$,  the space of functions $u \in W^{1,0}_r(Q)$ with $ u_t \in (W^{1,0}_{r'}(Q))'$, equipped with the norm
\[
\|u\|_{\H^{r,1}(Q)} = \|u\|_{W^{1,0}_r(Q)} + \|u_t\|_{(W^{1,0}_{r'}(Q))'}.
\]
For a given Banach space $X$, $L^p((0,T), X)$ and $C^{0,\theta}([0,T], X)$ will denote the usual Lebesgue and H\"older parabolic spaces respectively.

We recall here some embedding results enjoyed by $\H^{r,1}(Q)$. 

\begin{prop}\label{parab_embed} The following embeddings hold.
\begin{enumerate}
\item[\textit{(i)}] If $1 < r < N+2$, then $\H^{r,1}(Q)$ is continuously embedded in $L^q(Q)$, for $1 \le q \le  \frac{(N+2)r}{N+2 - r}$.
\item[\textit{(ii)}] If $ r \ge N+2$, then $\H^{r,1}(Q)$ is continuously embedded in $L^q(Q)$, for all $1 \le q < \infty$.
\item[\textit{(iii)}] If $ r > N+2$, then there exist $\nu, \theta > 0$ such that $\H^{r,1}(Q)$ is continuously embedded in $C^{0,\nu}([0,T], C^{0,\theta}(\T))$.
\end{enumerate}
\end{prop}

\begin{proof} For {\it i)} and {\it ii)}, see \cite[Theorem A.3]{MPR}. Such embeddings are indeed verified in the case $Q = \RsetN \times (0,T)$; in order to derive them in the case $Q = \T \times (0,T)$, it is sufficient to observe that a linear operator extending $\H^{r,1}(\T \times (0,T))$ to $\H^{r,1}(\RsetN \times (0,T))$ continuously can be constructed by standard methods. As for {\it iii)}, one may reason as in \cite[Proposition 3.5]{MPR}.
\end{proof}

When $1 < r < N+2$,  $\H^{r,1}(Q)$ can be compactly embedded in $L^q(Q)$  under the stronger assumption $1 \le q <  \frac{(N+2)r}{N+2 - r}$. 

\begin{prop}\label{parab_embed_comp} If $1 < r < N+2$, then $\H^{r,1}(Q)$ is compactly embedded in $L^q(Q)$, for all $1 \le q <  \frac{(N+2)r}{N+2 - r}$.
\end{prop}

\begin{proof} The proof relies on the so-called Aubin-Lions-Simon Lemma. 

Let $1 < r < N+2$. Note first that $W^{1,r'}(\T)$ is reflexive and separable. Therefore, $L^r((0,T), (W^{1,r'}(\T)')$ is isomorphic to $(L^{r'}((0,T), W^{1,r'}(\T)))'$, and the latter space coincides with $(W^{1,0}_{r'}(Q))'$. Since $W^{1,0}_{r}(Q)$ coincides with $L^{r}((0,T), W^{1,r}(\T))$, we have that the space $\H^{r,1}(Q)$ is isomorphic to
\[
E = \{ u \in L^{r}((0,T), W^{1,r}(\T)), \,  u_t \in L^r((0,T), (W^{1,r'}(\T)')\}.
\]
As $W^{1,r}(\T)$ is compactly embedded in $L^r(\T)$, and $L^r(\T)$ is continuously embedded in $(W^{1,r'}(\T)'$, the Aubin-Lions-Simon lemma (see \cite{Simon}) states that $E$ is compactly embedded in $L^r(Q)$ (this results actually holds for any $r \ge N+2$). Hence,  $\H^{r,1}(Q)$ is compactly embedded in $L^r(Q)$ and the result is for  $1\leq q \leq r$.

Let now $u_n$ be a bounded sequence in $\H^{r,1}(Q)$; we may extract a subsequence $u_{n_k}$ that converges to $u$ strongly in $L^r(Q)$. For any $r < q < \frac{(N+2)r}{N+2 - r}$,  by interpolation, there exists $0<\theta<1$ such that 
\[
\|u_{n_k}- u_{n_j}\|_{L^q(Q)} \le \|u_{n_k} - u_{n_j}\|^\theta_{L^r(Q)} \|u_{n_k} - u_{n_j}\|^{1-\theta}_{L^{\frac{(N+2)r}{N+2 - r}}(Q) }\to 0,
\]
as $j, k \to \infty$, being $u_{n_k}\in\H^{r,1}(Q)\subset L^{\frac{(N+2)r}{N+2 - r}}(Q)$  by Proposition \ref{parab_embed}, so $u_{n_k}$ converges also in $L^q(Q)$.
\end{proof}

The following is a classical H\"older regularity result for Hamilton-Jacobi equations with quadratic or subquadratic Hamiltonians.

\begin{prop}\label{hjb_regularity} Let $\Omega$ be a bounded domain in $\mathbb R^N$. Suppose that $v$ is a classical solution of
\[
-v_t - \Delta v + \widetilde H(\nabla v) = W(x,t) \quad \text{in $\widetilde Q = \Omega \times (0, \tau)$,} \ \tau>0
\]
and that for some $K > 0$,
\[
\begin{split}
\bullet  \quad & \|v\|_{L^\infty(\widetilde Q)}, \|W\|_{L^\infty(\widetilde Q)}, \|v(\cdot, \tau)\|_{C^2(\Omega)} \le K, \\
\bullet  \quad & |\widetilde H(p)| \le K(|p|^2 + 1) \text{ for all $p\in \mathbb R^N$}.
\end{split}
\]

Then, for all $\Omega' \subset \subset \Omega$, there exists $C > 0$, $\beta > 0$ (depending on $K, \Omega'$, but not on $\tau$),  such that
\[
\|v\|_{C^{1,\beta}(\Omega' \times (0, \tau))} \le C.
\]
\end{prop}

\begin{proof} See \cite[Theorem V.3.1]{lady}.
\end{proof}

\subsection{Regularity of the Fokker-Planck equation} 

In what follows, let $m \in L^1(Q)$ and $A$ be a measurable vector field such that $m$ is a weak solution of  the Fokker-Planck equation 
\begin{equation}\label{FKP}
\begin{cases}
m_t - \Delta m + \diverg(A \, m)  = 0 & \text{in $Q$,}\\
m(x,0) = m_0(x) & \text{on $\T$}.
\end{cases}
\end{equation}
Suppose  $|A|^{\gamma'} m \in L^1(Q)$, for a  $\gamma'>1$, and set $E$ to be the quantity \[E:= \int_Q |A|^{\gamma'} m\, dxdt. \]
In this section we will state some regularity results and a-priori estimates for $m$, depending on $E$, inspired by Metafune, Pallara, Rhandi \cite{MPR}, that will be used in the sequel. We stress that our aim here is to obtain regularity of $m$ in terms of $|A|m^{1/\gamma'}$, rather than in terms of $|A|$ itself. The former quantity is indeed associated to the energy of the system. The exponent $\gamma'$ and $\alpha$ used in this section are not necessarily linked  with the ones given by the hypotheses on $H$ and $f$ (unless otherwise specified).

Note that, by setting $w:=Am$ on the set $\{m > 0 \}$, and $w \equiv 0$ where $m$ vanishes, the couple $(m,w)$ solves
\[
\begin{cases}
m_t - \Delta m + \diverg(w)  = 0 & \text{in $Q$,}\\
m(x,0) = m_0(x) & \text{on $\T$},
\end{cases}
\]
and $E$ can be rewritten as $E= \int_Q |w/m|^{\gamma'} m\, dxdt.$

\begin{prop}\label{GN}  Let $m \in L^p(Q)$, for some $p> 1$, be  a weak solution of \eqref{FKP}, and $r$ be such that
\begin{equation}\label{rdef}
\frac{1}{r} := \frac{1}{\gamma'} + \left(1- \frac{1}{\gamma'} \right) \frac{1}{p}.
\end{equation}

Then, there exist $C > 0$, depending on $T, p, N, \gamma'$ and $\|\nabla m_0\|_{L^\infty(\T)}$, such that
\begin{equation}\label{Hnorm}
\|m\|_{\H^{r,1}(Q)} \le C(E^{1/\gamma'} \|m \|_{L^{p}(Q)}^{1/\gamma} + 1 ),
\end{equation}
where $\gamma$ is the conjugate exponent of $\gamma'$.
\end{prop}

\begin{proof} We first assume that $m$ is a smooth function; we will remove this requirement at the end of the proof. Let $\varphi \in C^{2,1}(\overline{Q})$ be a test function such that $\varphi(\cdot, T) = 0$. Then, being $m$ a weak solution,
\[
\int_Q m(-\varphi_t - \Delta \varphi - A \cdot \nabla \varphi)\, dxdt = \int_\T m_0(x) \varphi(x,0) dx.
\]
Hence,
\[
\left| \int_Q m(-\varphi_t - \Delta \varphi)\, dxdt \right| \le \left| \int_\T m_0(x) \varphi(x,0) dx \right| + \int_Q |A|m^{1/\gamma'} m^{1-1/\gamma'}|\nabla \varphi|\, dxdt,
\]
and, applying  twice H\"older inequality in the last term, we have
\begin{equation}\label{eq1}
\left| \int_Q m(-\varphi_t - \Delta \varphi)\, dxdt \right| \le \left| \int_\T m_0(x) \varphi(x,0) dx \right| + E^{1/\gamma'} \|m \|_{L^{p}(Q)}^{1/\gamma} \|\nabla \varphi \|_{L^{r'}(Q)},
\end{equation}
where, $r$ is as in \eqref{rdef}.

Let now $i = 1, \ldots, N$ be fixed, and $\psi \in C^{2,1}(Q)$ be the solution of 
\begin{equation}\label{eqpsi}
\begin{cases}
-\psi_t - \Delta \psi  = |\partial_{x_i} m|^{r-2}\partial_{x_i} m & \text{in $Q$,}\\
\psi(x,T) = 0 & \text{on $\T$}.
\end{cases}
\end{equation}
Note that, by standard elliptic regularity,
\begin{equation}\label{eq11}
\| \psi \|_{W^{2,1}_{r'}(Q)} \le C \|\, |\partial_{x_i} m|^{r-1} \, \|_{L^{r'}(Q)} = C \| \partial_{x_i} m  \|_{L^{r}(Q)}^{r-1},
\end{equation}
where $r'$ is the conjugate exponent of $r$.
Moreover, $\psi(x,0) = -\int_0^T \psi_t(x,s) \, ds$ on $\T$, therefore, using H\"older inequality,
\begin{align}\label{eq2}
\left| \int_\T \partial_{x_i}m_0(x) \psi(x,0) \, dx \right| &\le \int_Q |\partial_{x_i}m_0(x) \psi_t(x,s) |\, dxds \le T^{1/r} \| \partial_{x_i}m_0 \|_{L^r(\T)} \| \psi_t \|_{L^{r'}(Q)} \nonumber \\ &\le C \| \psi \|_{W^{2,1}_{r'}(Q)}.
\end{align}
If we now let $\varphi = \partial_{x_i}\psi$ in \eqref{eq1}, integrating by parts, we obtain
\[
\left| \int_Q \partial_{x_i}m (-\psi_t - \Delta \psi)\, dxdt \right| \le \left| \int_\T \partial_{x_i} m_0(x) \psi(x,0) dx \right| + E^{1/\gamma'} \|m \|_{L^{p}(Q)}^{1/\gamma} \|\nabla (\partial_{x_i} \psi)  \|_{L^{r'}(Q)},
\]
and by \eqref{eq2} and the fact that $\psi$ solves \eqref{eqpsi},
\[
\int_Q |\partial_{x_i}m|^r \, dxdt  = \left| \int_Q \partial_{x_i}m (-\psi_t - \Delta \psi)\, dxdt \right| \le C \| \psi \|_{W^{2,1}_{r'}(Q)} (1 + E^{1/\gamma'} \|m \|_{L^{p}(Q)}^{1/\gamma}).
\]
Hence, using again \eqref{eq11},
\begin{equation}\label{eq3}
\|\nabla m\|^r_{L^r(Q)} = \int_Q |\nabla m|^r \, dxdt \le C(E^{r/\gamma'} \|m \|_{L^{p}(Q)}^{r/\gamma} + 1 ).
\end{equation}
By Poincar\'e inequality and \eqref{eq11}, since $\frac{1}{|Q|}\int_Q m \, dx dt = 1$, we infer
\begin{equation}\label{eq31}
\|m\|_{L^r(Q)} \le \|m - 1\|_{L^r(Q)} + T \le C(E^{1/\gamma'} \|m \|_{L^{p}(Q)}^{1/\gamma} + 1 ).
\end{equation}

Finally, we multiply the Fokker-Planck equation by any test function (that may not vanish at time $T$) $\varphi \in C^{2,1}(\overline{Q})$ to get
\begin{align*}
\left| \int_Q m_t \varphi \, dxdt \right| &\le \left| \int_Q \nabla m \cdot \nabla \varphi \, dx dt \right| + E^{1/\gamma'} \|m \|_{L^{p}(Q)}^{1/\gamma} \|\nabla \varphi \|_{L^{r'}(Q)} \\ &\le (\|\nabla m\|_{L^r(Q)}+ E^{1/\gamma'} \|m \|_{L^{p}(Q)}^{1/\gamma}) \|\nabla \varphi \|_{L^{r'}(Q)},
\end{align*}
where  H\"older inequality has been used for the integral of $A\cdot \nabla \varphi m$ as in \eqref{eq1}. Hence 
$$\| m_t\|_{(W^{1,r'}(Q))'} \le C(E^{1/\gamma'} \|m \|_{L^{p}(Q)}^{1/\gamma} + 1 ),$$
 and by \eqref{eq3} and \eqref{eq31} we conclude.

If $m$ is not smooth, consider a regularised sequence $(m_n, w_n) := (m \star \xi_n, w \star \xi_n)$, where $ \xi_n$ is a (space-time) smoothing kernel. Then, $(m_n)_t - \Delta m_n + \diverg(w_n)  = 0$ holds, and \eqref{Hnorm} is verified with $E_n = \int_Q |w_n/m_n|^{\gamma'} m_n\, dxdt$. Since $E_n \to E$ as $n \to \infty$ (see \cite{CardaGraber}, Lemma 2.7), we obtain \eqref{Hnorm} in the general case.

\end{proof}

The following is a crucial estimate that links the energy term $\left( \int_Q m^{\alpha+1} \, dxdt \right)^\delta$, for a $\delta>1$, with the quantity $E$. The assumption $\alpha <  \gamma'/N$ basically allows to apply the Gagliardo-Nirenberg and interpolation inequalities.

\begin{prop}\label{GN2} 
Suppose that $m \in L^{\alpha+1}(Q)$ for some $0\leq \alpha <  \gamma'/N$. Then, there exist $C > 0$ and $\delta > 1$, depending on $T, \alpha, N, \gamma$ and $\|\nabla m_0\|_{L^\infty(\T)}$, such that
\begin{equation}\label{EI}
\left( \int_Q m^{\alpha+1} \, dxdt \right)^\delta \le C \left( E + 1\right).
\end{equation}
\end{prop}

\begin{proof}
Since $\|m\|_{L^1(Q)} = T$, if $\alpha = 0$, we are done.  Suppose then $\alpha>0$.
Let $r$ be defined as in \eqref{rdef} for $p=\alpha+1$. By the Gagliardo-Nirenberg inequality, there exists $C > 0$ such that for all $t \in (0,T)$,
\[
\int_\T m^\eta (x,t)\, dx \le C\left(\int_\T |\nabla m(x,t)|^r \, dx + 1 \right),
\]
where $\eta: = r(N+1)/N$ and we used the fact that  the $L^1(\T)$ norm of $m(\cdot, t)$ is equal to one for all $t$. If we now integrate the previous equality over $(0,T)$, in view of \eqref{Hnorm}, it follows that
\[
\|m\|^{(N+1)/N}_{L^{\eta}(Q)}  \le C( \|m\|^{1/\gamma}_{L^{\alpha+1}(Q)} E^{1/\gamma'} + 1).
\]
It is now crucial to observe that $\alpha <  \gamma'/N$ implies $1 < \alpha+1< \eta$. Arguing by interpolation, there exists some $0 < \theta < 1$ such that $\|m\|_{L^{\alpha+1}(Q)} \le C\|m\|_{L^{\eta}(Q)}^\theta$, so
\[
\|m\|^{(N+1)/(\theta N)}_{L^{\alpha+1}(Q)}  \le C( \|m\|^{1/\gamma}_{L^{\alpha+1}(Q)} E^{1/\gamma'} + 1),
\]
which implies
\[
\|m\|^{1 + \gamma' (N+1)/(\theta N) - \gamma'}_{L^{\alpha+1}(Q)}  \le C(E + 1).
\]
The inequality \eqref{EI} then follows setting $\delta:= \frac{ 1+ \gamma' (N+1)/(\theta N) - \gamma'   } {\alpha+1 }$. Note that   $\delta>1$ since $ \gamma' (N+1)/(\theta N) - \gamma' \ge \gamma'/N > \alpha$.

\end{proof}

We next state a proposition that gives a bound on the $\H^{r,1}(Q)$ norm of the solution of the Fokker-Planck equation in terms of $E$, for certain values of $r$ that depends on $\gamma'$ and $N$. 

\begin{prop}\label{Calphaest} Suppose that $E \le K$ and $\|m\|_{L^{p_0}(Q)} \le K$, for some $K > 0$,
\begin{equation*}\label{p0range}
\begin{cases}
p_0 \in \left(1, \frac{N+2}{N+2-\gamma'}\right) & \text{if $\gamma' < N+2$}, \\
p_0 \in (1, +\infty) & \text{if $\gamma' \geq  N+2$}.
\end{cases}
\end{equation*}

Let $r$ be such that 
\begin{equation}\label{rrange}
\begin{cases}
r \in \left(1, \frac{N+2}{N+3-\gamma'}\right) & \text{if $\gamma' < N+2$}, \\
r \in (1, \gamma') & \text{if $\gamma' \geq  N+2$}.
\end{cases}
\end{equation}
Then, there exists $C > 0$ depending on $K, r, T, N, \gamma'$ such that 
\[
\|m\|_{\H^{r,1}(Q)} \le C.
\]
\end{prop}

\begin{proof}  Let the sequences $p_n, r_n$ be defined by induction as follows: for a given $p_n$, let $r_n$ be such that
\[
\frac{1}{r_n} := \frac{1}{\gamma'} + \left(1- \frac{1}{\gamma'} \right) \frac{1}{p_n}> \frac{1}{\gamma'} .
\]
Moreover, $p_{n+1} :=  \frac{(N+2)r_n}{N+2 - r_n} $, i.e.
\[
\frac{1}{p_{n+1}} = \frac{1}{r_n}  - \frac{1}{N+2}  =  \frac{1}{\gamma'}  - \frac{1}{N+2} + \left(1- \frac{1}{\gamma'} \right) \frac{1}{p_n}.
\]

We use here the convention  $ \frac{N+2-\gamma'}{(N+2)}=0$, when $\gamma'\geq N+2$. Since $1 < p_0 < \frac{N+2}{N+2-\gamma'}$, $r_n, p_n$ are  increasing sequences. Indeed, we have
$$
\frac{p_{n}}{p_{n+1}}=  \frac{N+2-\gamma'}{\gamma' (N+2)} p_{n} -\frac{1}{\gamma'}+ 1<1$$
as soon as $\gamma'\geq N+2$,  or as $p_{n}< \frac{N+2}{N+2-\gamma'}$.

In the case  $\gamma' <  N+2$, $p_n$ converges to $\frac{N+2}{N+2-\gamma'}$, while $r_n$ converges to $\frac{N+2}{N+3-\gamma'}$.
By Proposition \ref{parab_embed} and Proposition \ref{GN}, we have that
\[
\|m\|_{L^{p_{n+1}}(Q)} \le C \|m\|_{\H^{r_n,1}(Q)} \le C_1(E^{1/\gamma'} \|m \|_{L^{p_n}(Q)}^{1/\gamma} + 1 ),
\]
so we obtain the assertion by iterating the last inequality a finite number of times.

As for the case $\gamma' \geq  N+2$, one argues in a similar way, with the difference that  $p_n \to +\infty$ and $r_n \to \gamma'$.

\end{proof}
Together with the embedding results  of Proposition  \ref{parab_embed} and Proposition \ref{parab_embed_comp}, Proposition \ref{Calphaest} allows to prove the strong convergence of a (sub)sequence of weak solutions $m_n$, as shown in the following  corollary.

\begin{cor}\label{comp} Suppose that $(m_n, A_n)$ is a sequence solving the Fokker-Planck equation \eqref{FKP} in the weak sense. Let 
\[
E_n:=\int_{Q} |A_n|^{\gamma'} m_n \,dxdt \le K
\]
and $\|m_n\|_{L^{p_0}(Q)} \le K$ for some $K > 0$, 
\begin{equation*}\label{p0range2}
\begin{cases}
p_0 \in \left(1, \frac{N+2}{N+2-\gamma'}\right) & \text{if $\gamma' < N+2$}, \\
p_0 \in (1, +\infty) & \text{if $\gamma' \geq  N+2$}.
\end{cases}
\end{equation*}
Then, up to subsequences, $m_n$ converges:
\begin{itemize} \item strongly in $L^p(Q)$ for any $p \in [1, \frac{N+2}{N+2-\gamma'})$, if $\gamma' < N+2$, 
\item strongly in $L^p(Q)$ for any $p \in [1, +\infty)$, if $\gamma' \geq N+2$, 
\item  in $C^{0,\theta}(Q)$ for some $\theta > 0$, if $\gamma' > N+2$. 
\end{itemize}
\end{cor}

\begin{proof} By Proposition \ref{Calphaest}, $m_n$ is bounded in ${\H^{r,1}(Q)}$, for all $r$ defined as in \eqref{rrange}. 

If $\gamma' < N+2$,  for any $p \in [1, \frac{N+2}{N+2-\gamma'})$, we can find  $\bar r\in  \left(1, \frac{N+2}{N+3-\gamma'}\right)$, $\bar r <\frac{N+2}{N+3-\gamma'}<N+2$, such that  $p \in [1, \frac{(N+2)\bar r}{N+2-\bar r})$ and $m_n$ is bounded in ${\H^{\bar r,1}(Q)}$. Hence it is sufficient to apply Proposition \ref{parab_embed_comp} to obtain
that $m_n$ converges strongly (up to subsequences) $L^p(Q)$.

If $\gamma' \geq  N+2$ then for any $p \in [1, +\infty)$ we can find  $\bar r \in (1, \gamma')$, $\bar  r< N+2$, such that  $p \in [1, \frac{(N+2)\bar r}{N+2-\bar r})$ and $m_n$ is bounded in ${\H^{\bar r,1}(Q)}$. Therefore we can apply Proposition \ref{parab_embed_comp} to obtain
that $m_n$ converges strongly (up to subsequences) in $L^p(Q)$.

If $\gamma' >  N+2$, we can find $\bar r \in (1, \gamma')$,  $\bar r>N+2$, such that $m_n$ is bounded in ${\H^{\bar r,1}(Q)}$. Hence by  Proposition \ref{parab_embed}, $m_n$ is bounded in some $C^{0,\nu}([0,T], C^{0,\theta'}(\T))$ (compactness follows by choosing $\theta < \min\{\theta',\nu\}$).

\end{proof}

\section{Existence of solutions}\label{existence}
In this section we discuss the existence of solution for the MFG system \eqref{MFG}. We begin giving the definition of weak solution.

Let $\alpha,\gamma$ be the exponents defined by the hypotheses on $H$ and $f$,
$\alpha < \frac{\gamma'}{N}.$ Let 
$$ q:=\left\{\begin{array}{ll} \frac{\gamma (\alpha+1)(1+N)}{\alpha N-\gamma}& \text{ if } \alpha > \frac{\gamma}{N}\\ +\infty &\text{ if } \alpha < \frac{\gamma}{N}.\end{array}\right.$$
\begin{defn}\label{def:weaksolMFG} We say that a pair $(u,m)\in L^q(Q) \times L^{\alpha+1}(Q)$ is a weak solution to \eqref{MFG}, if 
\begin{itemize}
\item[(i)] the following integrability conditions hold:
$$ \nabla u\in L^\gamma(Q), \;  mL(\nabla H(\nabla u))\in L^1 (Q)
\quad{\rm and }\quad m \nabla H(\nabla u)\in L^1 (Q).
$$ 

\item[(ii)] Equation \eqref{MFG}-(i) holds in the following sense: inequality
\begin{equation}\label{eq:distrib}
 \quad -u_t-\Delta u +H(\nabla u)\leq -  f(x,m) \quad {\rm in }\; Q, 
\end{equation}with $u(\cdot,T)\leq u_T$, holds in the sense of distributions,  

\item[(iii)] Equation \eqref{MFG}-(ii) holds:  
\begin{equation}\label{eqcontdef}
 \quad m_t-\Delta m-{\rm div}(m \nabla H(\nabla u)))= 0\  {\rm in }\; Q, \quad m(0)=m_0
\end{equation}
in the sense of distributions,

\item[(iv)] The following equality holds: 
\begin{align}\label{defcondsup} 
 \int_Q m(x,t)(f(x,m(x,t))&+ L( \nabla H(\nabla u)(x,t)) )dxdt\\
 &+ \int_\T (m(x,T)u_T(x)-m_0(x)u(x,0))dx=0. \nonumber
\end{align}
\end{itemize}
\end{defn}


Using the estimates obtained for the solution of the Fokker-Planck equation, we are able to prove that the energy functional $\Ecal$, defined as in \eqref{defE},  is bounded from below over the set $\mathcal K$.

\begin{lem}\label{Eboundbelow} There exists $c \in \Rset$ such that
\[
c = \inf_{(m,w) \in \Kcal} \Ecal(m, w).
\]

Moreover, suppose that for a sequence $(m^n,w^n)\in \Kcal$, there exists $e \in \Rset$ such that $\Ecal(m^n, w^n) \le e$ for all $n\in\mathbb N$. Then, for some $c_1$ (depending on $e$), and for all $n\in\mathbb N$
\[
\int_Q (m^n)^{\alpha+1} dx dt +  \int_Q \frac{|w^n|^{\gamma'}}{(m^n)^{\gamma'-1}} dx dt \le c_1.
\]
\end{lem}

\begin{proof} Let $(m,w) \in \Kcal$, hence $m\in L^{\alpha+1}(Q)$, $w\in L^1(Q)$. Since $m$ is a weak solution of the Fokker-Planck equation with drift $A = w/m$, we may apply Proposition \ref{GN2} to infer the existence of $C > 0$ and $\delta > 1$ (depending on the data), such that
\begin{equation}\label{eq41}
\left(\int_Q m^{\alpha+1} dx dt\right)^\delta \le C \left(\int_Q \frac{|w|^{\gamma'}}{m^{\gamma'-1}}dx dt + 1\right).
\end{equation}
Moreover, by \eqref{Lass} and \eqref{Fass}, we have that

\begin{align}\label{eq42}
\Ecal(m,w)& \ge C_L^{-1} \int_Q \frac{|w|^{\gamma'}}{m^{\gamma'-1}} dx dt - C_F \int_Q m^{\alpha+1} dx dt - \|u_T\|_{L^\infty(\T)} - C_F T \\& \ge C\left(\int_Q m^{\alpha+1} dx dt\right)^\delta - C_F \int_Q m^{\alpha+1} dx dt - \|u_T\|_{L^\infty(\T)} - C,\nonumber 
\end{align}
which has a finite infimum, since $\delta>1$.

In order to prove the second assertion, it suffices to use \eqref{eq41} for the sequence $(m^n,w^n)$, rewriting \eqref{eq42} as
\begin{align*}
C_L^{-1} \int_Q \frac{|w^n|^{\gamma'}}{(m^n)^{\gamma'-1}} dx dt &\le \Ecal(m^n,w^n)  + C_F \int_Q (m^n)^{\alpha+1} dx dt + \|u_T\|_{L^\infty(\T)} + C_F T \\ &\le C + C \left(\int_Q \frac{|w^n|^{\gamma'}}{(m^n)^{\gamma'-1}}dx dt + 1\right)^{1/\delta},
\end{align*}
and by \eqref{eq41} we conclude.
\end{proof}

We prove now that, up to subsequences, a minimising sequence of $\Ecal$ converges (in different spaces, according to the value of $\gamma'$) to a minimiser.

\begin{lem}\label{min_seq_lem} Let $(m^n, w^n) \in \Kcal$ be a minimising sequence, that is
\begin{equation}\label{min_seq}
\Ecal(m^n, w^n) \xrightarrow{n \to \infty} c := \inf_{(m,w) \in \Kcal} \Ecal(m, w).
\end{equation}
Then, up to subsequences, $m^n \to \bar{m}$ strongly  in $L^p(Q)$ and $w^n\rightharpoonup\bar w$ weakly in  $L^{\frac {\gamma' p}{\gamma' +p-1}}(Q)$, where:
\begin{itemize}
\item  $p \in  [1,\frac{N+2}{N+2-\gamma'})$, if $ 1 < \gamma' \le N+2$;
\item  $p \in  [1,+\infty)$, if $  \gamma' \geq  N+2$.
\end{itemize}
Moreover, if $\gamma' > N+2$,  then $m^n \to \bar{m}$ in $C^{0,\theta}(Q)$ for some $\theta > 0$.
 
In particular, we can always take $p=\alpha+1$ in the above statement, and,  up to subsequences, $m^n \to \bar m$ a.e. in $Q$. Then, the couple $(\bar m, \bar w)$ is a minimiser of $\Ecal$ in $\Kcal$.

\end{lem}

\begin{proof} Let $(m^n, w^n) \in \Kcal$ be a minimising sequence. By choosing $n$ large enough, $\Ecal(m^n, w^n) \le c+1$, and Lemma \ref{Eboundbelow} implies that
\[
 \int_Q \frac{|w^n|^{\gamma'}}{(m^n)^{\gamma'-1}} dx dt \le c_1
\]
for some $c_1 > 0$. Note that since $m^n$ is  a weak solution to the Fokker-Planck equation with drift $A^n = w^n/m^n$, 
\begin{equation}\label{Aintegrability}
\int_Q |A^n|^{\gamma'} m^n \, dx dt = \int_Q \frac{|w^n|^{\gamma'}}{(m^n)^{\gamma'-1}} \, dx dt \le c_1.
\end{equation}
The thesis for $m^n$ follows applying Corollary \ref{comp}. In all cases, $\bar m \in L^{\alpha+1}(Q)$ (since $\alpha +1 < \frac{N+2}{N+2-\gamma'}$) and, up to subsequences, $m^n \to \bar m$ a.e. in $Q$. 

Concerning $w^n$, for all $p$ such  that $m_n \to \bar{m}$ strongly in $L^p(Q)$, by H\"older inequality
$$
\int_Q |w^n|^{\frac {\gamma' p}{\gamma' +p-1}}=\int_{\{m^n>0\}} |w^n|^{\frac {\gamma' p}{\gamma' +p-1}}\leq \|m^n\|_{L^{p}(Q)}^{\frac {\gamma' -1}{\gamma' +p-1}} \left(\int_Q \frac{|w^n|^{\gamma'}}{(m^n)^{\gamma'-1}} \, dx dt\right)^{\frac{p}{\gamma'+p-1}}\leq C.
$$
Hence, $w^n$ converges weakly to $\bar w$ in $L^{\frac {\gamma' p}{\gamma' +p-1}}(Q)$ and we can take in particular $p=\alpha+1$. The fact that the limit $(\bar m, \bar w)$ is a minimiser readily follows by lower semi-continuity of $\mathcal E$.
\end{proof}

\subsection{A convex problem}\label{S31}

In order to find a link between the minimiser of $\Ecal$ and the solution of \eqref{MFG}, being the energy $\Ecal $  not convex in $(m,w)$ due to the presence of the term $-\int_Q F(x,m)dxdt$, we convexify it by adding a term that vanishes in $\bar m$, the limit of the minimising sequence $m^n$. Therefore, let us define  
\[
\overline{\Ecal}(m, w) = \Ecal(m, w) + \int_Q G(x,t,m) dx dt,
\]
where
\[
G(x,t,m) := \frac{c_f+1}{\alpha(\alpha + 1)}\left[ (m+1)^{\alpha+1} - (\bar{m}(x,t) + 1)^{\alpha+1}\right] - \frac{c_f+1}{\alpha}(\bar{m}(x,t) + 1)^\alpha(m - \bar{m}(x,t)).
\]
Then,
\[
\begin{split}
& g(x,t,m) := \partial_m G(x,t,m) =   \frac{c_f+1}{\alpha} (m+1)^{\alpha} - \frac{c_f+1}{\alpha}(\bar{m}(x,t) + 1)^\alpha , \\
& \partial_m g(x,t,m) =\partial_{mm} G(x,t,m) =(c_f+1)(m+1)^{\alpha-1}.
\end{split}
\]
Note that $G(x,t, \bar{m}(t,x)) = \partial_m G(x,t, \bar{m}(t,x)) = 0$ for all $(x,t)\in Q$, and $\partial_{mm} G(x,t, m) \ge 0$ for all $m\geq0$, so that $G(x,m) \ge 0$ everywhere. Moreover,
\begin{multline}\label{fgboundvar}
 \frac{c_f}{\alpha} - f(0) + \frac{1}{\alpha}(m+1)^\alpha - \frac{c_f+1}{\alpha} (\bar{m}(x,t) + 1)^\alpha \le \\ (-f+g)(x, t, m) \le \frac{c_f+1}{\alpha}(m+1)^\alpha - \frac{c_f+1}{\alpha} (\bar{m}(x,t) + 1)^\alpha.
\end{multline}

\begin{lem}\label{Ebar_convex} $\overline{\Ecal}$ is convex on $\Kcal$,  strictly convex  with respect to $m$, that is
\[
\overline{\Ecal}(\tau m + (1- \tau)\mu, \tau w + (1-\tau)v) \le \tau \overline{\Ecal}(m, w) + (1-\tau)\overline{\Ecal}(\mu, v) - \frac{1}{2}\tau(1-\tau) \int_Q \psi(m, \mu) \, dxdt
\]
for all $\tau \in [0,1]$, $(m,w),(\mu,v) \in \Kcal$, where $\psi(x, y) = \min \{(x+1)^{\alpha-1}, (y+1)^{\alpha-1} \}(x-y)^2$.
\end{lem}

\begin{proof} It is standard to show that $(m,w) \mapsto \int_Q m L\left(-\frac{w}{m}\right)dxdt + \int_\T u_T(x) m(x,T) \, dx$ is convex. It is then sufficient to note that $\partial_{mm} (-F+G)(x,m) = -\partial_m f(x,m) + (c_f+1)(m+1)^{\alpha-1} \ge (m+1)^{\alpha-1}$, because of \eqref{fass}.
\end{proof}

We consider now the MFG system associated to $\overline{\Ecal}$:
\begin{equation}\label{MFGbar}
\begin{cases}
-u_t - \Delta u + H(\nabla u) = -f(x,m) + g(x,t,m), \\
m_t - \Delta m - \diverg(\nabla H(\nabla u)\, m)  = 0 & \text{in $Q$,} \\
m(x,0) = m_0(x), \quad u(x,T) = u_T(x) & \text{on $\T$}.
\end{cases}
\end{equation}

\begin{prop}\label{mfgbarex} If  $\gamma' > 1$, there exists a weak solution $(\tilde{u}, \tilde{m})$ to the MFG system \eqref{MFGbar}, such that the couple $(\tilde m, \tilde m\nabla H(\nabla \tilde u))$ is a minimiser of the (convex) energy functional $\overline{\Ecal}$ in $\Kcal$. Moreover, if $\gamma' > N+2$, such a solution is classical.
\end{prop}

\begin{proof}


{\bf 1.}  We consider first the case $1< \gamma' \le N+2$. 

 In order to apply the standard  variational theory for convex systems, see \cite{CGPT}, we need to truncate $\bar m$ because in  \cite{CGPT} the  estimate \eqref{fgboundvar} is given with left and right bounds independent from $(x,t)$. 
 
{\bf 1.1} Let us, therefore, consider the truncated function $\bar m_M$ defined for $M>0$, as 
$$
\bar m_M(x,t):=\left\{ \begin{array} {ll}
\bar m(x,t) & \forall (x,t)\in Q \ \ \text {s.t}  \ \ 0\leq \bar m(x,t)\leq M,\\
M & \text{otherwise} 
\end{array}
\right. 
$$
and the truncated energy functional 
\[
\overline{\Ecal}_M(m, w): = \Ecal(m, w) + \int_Q G_M(x,t,m) dx dt,
\]
where
\[
G_M(x,t, m) : = \frac{c_f+1}{\alpha(\alpha + 1)}\left[ (m+1)^{\alpha+1} - (\bar{m}_M(x,t) + 1)^{\alpha+1}\right] - \frac{c_f+1}{\alpha}(\bar{m}_M(x,t) + 1)^\alpha(m - \bar{m}_M(x,t)).
\]
Then, defining
\[
 g_M(x,t,m) := \partial_m G_M(x,t,m) =   \frac{c_f+1}{\alpha} (m+1)^{\alpha} - \frac{c_f+1}{\alpha}(\bar{m}_M(x,t) + 1)^\alpha , 
\]
we have, in particular, that 
\begin{equation*}\label{fgbound}
 \frac{c_f}{\alpha} - f(0) + \frac{1}{\alpha}(m+1)^\alpha - \frac{c_f+1}{\alpha} (M + 1)^\alpha \le (-f+g_M)(x,t, m) \le \frac{c_f+1}{\alpha}(m+1)^\alpha - \frac{c_f+1}{\alpha} (M + 1)^\alpha.
\end{equation*}
Thus $\overline{\Ecal}_M(m, w)$ is strongly  strictly  convex w.r.t. $m$ and it is precisely a functional of the type studied in \cite{CGPT}.   Therefore, for all $M>0$, Theorem 6.2 in \cite{CGPT} gives us a weak solution $(u_M, m_M)$ such that $u_M$ is bounded below by a constant depending on $\|u_T\|_{C^2}$ and on $\|H(\nabla u_T)\|_{\infty}$. 

{\bf 1.2.} We have now to show the stability of solutions with respect to this approximation.  We follow Section 6.4 in \cite{CGPT}. Note that it is enough to set $A=Id$ in their second order  MFG system in order to obtain our truncated convex problem. On the one side, we are in a simpler case than the one in  \cite{CGPT}, since we are only approximating the coupling function $-f+g$ with the sequence $-f+ g_M$, while all the other data are not approximated. On the other,  even if $-f+ g_M$ converges locally uniformly to $-f+g$,  the limit $-f+g$ satisfies the more general inequality \eqref{fgboundvar}, where the right and left bounds  depend also on $(x,t)$ (not only in $m$).  However, we have the additional information that they are  bounded in $L^{1+\frac{1}{\alpha}} (Q)$, since $\bar m \in L^{\alpha+1}(Q) $. Note that the fact that the sequence $-f+ g_M$ depends also on time does not add any difficulty as stated in the introduction of \cite{CGPT}. 

 Let $w_M:=- m_M \nabla H(\nabla u_M)$. By Theorem 6.4 in \cite{CGPT}, the couple $(m_M,w_M)$ is a minimiser for $\overline \Ecal_M 
$. Proceeding as in  \cite{CGPT} (or as we did for the minimising sequence $(m^n, w^n)$ in Lemma \ref{min_seq_lem}),  we can prove that 
$$
\|m_M\|_{L^{\alpha+1}(Q)} +\|w_M\|_{L^{\frac {\gamma' (\alpha +1)}{\gamma' +\alpha}}(Q)}+ \left\|\frac {|w_M|^{\gamma'}}{(m_M)^{\gamma'-1}}\right\|_{L^{1}(Q)} \leq C
$$
and for all $M>0$, $t\to m_M(t)$ are uniformly H\"older continuous in $\mathcal P (\T)$. Hence, up to a subsequence, $(m_M,w_M)$ converges weakly in $L^{\alpha+1}(Q)\times L^{\frac {\gamma' (\alpha +1)}{\gamma' +\alpha}}(Q)$ to some $(\tilde m,\tilde w)$ and  $m_M(t)$ converges to $\tilde m (t)$ in $C^0([0,T], \mathcal P(\T))$. It follows that  $(\tilde m,\tilde w)$ satisfy $\tilde w ^{\gamma'} \tilde m^{ 1 - \gamma'} \in L^1(Q)$ and for all $ \varphi \in C_0^\infty(\T \times [0,T))$
\begin{align*}
0=&\lim_{M\to+\infty }\int_Q (m_M  \varphi_t + w_M\cdot \nabla \varphi + m_M  \Delta \varphi ) dx dt+  \int_{\T}m_0(x) \varphi(x,0) dx \\
=&\int_Q (\tilde m  \varphi_t + \tilde w\cdot \nabla \varphi + \tilde m  \Delta \varphi ) dx dt+  \int_{\T}m_0(x) \varphi(x,0) dx.
\end{align*}
Hence $(\tilde m,\tilde w)$ satisfies  \eqref{kcalconstraint}. 

 We claim that
\begin{equation}\label{estim}
\limsup_{M\to+\infty} \inf_{(m,w)\in\mathcal K} \overline \Ecal_M (m,w) \leq \inf _{(m,w)\in\mathcal K} \overline \Ecal (m,w).
\end{equation}
Indeed, 
\begin{multline*}
\limsup_{M\to+\infty} \inf_{(m,w)\in\mathcal K} \overline \Ecal_M (m,w)\\ \leq \limsup _{M\to+\infty} \int_Q m L\left(-\frac{w}{m}\right) - F(x,m) + G_M(x,t,m)\, dx dt + \int_\T u_T(x) m(x,T) \, dx,
\end{multline*}
for all $(m,w)\in\mathcal K$. Now, the locally uniform convergence of $G_M$ to $G$ gives, by dominated convergence theorem, 
$$
  \limsup _{M\to+\infty} \int_Q  G_M(x,t,m)\, dx dt=\lim _{M\to+\infty} \int_Q  G_M(x,t,m)\, dx dt=  \int_Q  G(x,t,m)\, dx dt.
$$
Hence, 
 \begin{align*}
\limsup _{M\to+\infty} \int_Q m L\left(-\frac{w}{m}\right) - F(x,m) + G_M(x,t,m)\, dx dt + \int_\T u_T(x) m(x,T) \, dx= \overline{\Ecal} (m,w)
\end{align*}
for all $(m,w)\in\mathcal K$. Thus, 
$$
\limsup_{M\to+\infty} \inf_{(m,w)\in\mathcal K} \overline \Ecal_M (m,w)\leq\overline{\Ecal} (m,w)
$$
for all $(m,w)\in\mathcal K$ and \eqref{estim} holds.

 Let us now show that
$$
\liminf _{M\to+\infty} \overline{\Ecal}_M(m_M, w_M) \geq \overline{\Ecal} (\tilde m,\tilde w).
$$
Note that, by convexity of $L$, we have 
$$
\liminf _{M\to+\infty} \int_Q m_M L\left(-\frac{w_M}{m_M}\right) dx dt \geq 
\int_Q \tilde m L\left(-\frac{\tilde w}{\tilde m}\right) dx dt, 
$$
by convexity of $-F+G_M$ and the locally uniform convergence of $G_M$ to $G$
$$
\liminf _{M\to+\infty} \int_Q (- F(x,m_M) + G_M(x,t,m_M))\, dx dt  \geq 
\int_Q ( - F(x,\tilde m) + G(x,t,\tilde m))\, dx dt, 
$$
and by the  convergence of $m_M(T)$
$$
\lim _{M\to+\infty}  \int_\T u_T(x) m_M(x,T) \, dx =  \int_\T u_T(x) \tilde m(x,T) \, dx.
$$
Hence, 
\begin{align*}
\liminf _{M\to+\infty} \overline{\Ecal}_M & (m_M, w_M) =\\ &\liminf _{M\to+\infty} \int_Q m_M L\left(-\frac{w_M}{m_M}\right) - F(x,m_M) + G_M(x,t,m_M)\, dx dt + \int_\T u_T(x) m_M(x,T) \, dx\\ &\geq 
\int_Q \tilde m L\left(-\frac{\tilde w}{\tilde m}\right) - F(x,\tilde m) + G(x,t,\tilde m)\, dx dt + \int_\T u_T(x) \tilde m(x,T) \, dx\\
&=\overline{\Ecal} (\tilde m,\tilde w).
\end{align*}

Therefore, thanks to \eqref{estim},  $(\tilde m,\tilde w)$ minimises $\overline{\Ecal}$,
$$
\lim _{M\to+\infty} \int_Q m_M L\left(-\frac{w_M}{m_M}\right) dx dt =
\int_Q \tilde m L\left(-\frac{\tilde w}{\tilde m}\right) dx dt
$$
and 
$$
\lim_{M\to+\infty} \int_Q (- F(x,m_M) + G_M(x,t,m_M))\, dx dt =
\int_Q ( - F(x,\tilde m) + G(x,t,\tilde m))\, dx dt.
$$
Since $-F+G_M$ are bounded below and $-F+G$ is strictly convex,   an argument using Young measures as  the one in \cite{CGPT}, gives us the strong convergence of $( m_M,w_M)$ to a minimiser $(\tilde m,\tilde w)$ of $\overline{\Ecal}$, in $L^{\alpha+1}(Q)\times L^{\frac {\gamma' (\alpha +1)}{\gamma' +\alpha}}(Q)$. Being the minimiser unique,  due to strict convexity of $\bar \Ecal$, the full sequence $( m_M,w_M)$ strongly converges to $(\tilde m,\tilde w)$. 

{\bf 1.3}  We are left to prove that  we can find  a weak solution for  \eqref{MFGbar} from the minimiser $(\tilde m,\tilde w)$.

 Let  $\beta_M(x,t):= (-f + g_M) (x,t, m_M(x,t)) $ on $Q$.  Thanks to the growth condition on $-f+g_M$ and the uniform bound of $m_M$ in $L^{\alpha+1}(Q)$, the sequence $\beta_M$ weakly converges in $L^{1+ \frac 1 \alpha}(Q)$ to $\tilde \beta$.

Being $u_M$ uniformly bounded by below, Theorem 3.3  in   \cite{CGPT} gives
\begin{equation}\label{boundu}
\|u_M\|_{L^\infty((0,T), L^q(\T))}+\|u_M\|_{L^\gamma(Q)}\leq C,
\end{equation}
Hence, up to a subsequence, $u_M$ weakly converges to $\tilde u$ in $L^q(Q)$. Moreover, proceeding as   in \cite{CGPT} we can prove that $\nabla u_M$ converges weakly to $\nabla \tilde u$ in $L^\gamma(Q)$. Hence by convexity of $H$ we have that $(\tilde u, \tilde \beta)$ satisfies 
$$
 -\tilde u_t-\Delta \tilde u +H(\nabla \tilde u)\leq \tilde \beta \quad {\rm in }\; Q, 
$$
in the sense of distributions.

By Lemma 5.3 in \cite{CGPT} (which does not involve the coupling function therefore it holds even for our limit functions), we have
\begin{equation}\label{ineq}
\left[ \int_{\T} \tilde m \tilde u\right]_0^{T} + \int_Q \tilde m \left(\tilde \beta  + L\left(-\frac{\tilde w}{\tilde m}\right)\right) \; \geq \; 0.
\end{equation}
Moreover, for all $M>0$, being $(u_M,m_M)$ a weak solution of  \eqref{MFG} with coupling function $-f+g_M$, we have that \eqref{defcondsup} is satisfied and for a.e. $(x,t)\in
Q$
$$
(-F+G_M)^*(x,t, \beta_M) + (-F+ G_M)(x, t, m_M)=\beta_M(x,t) m_M(x,t)= (-f +g_M)(x,t,m_M) m_M(x,t),
$$
(here, $(-F+G_M)^*$ is the Legendre transform of $-F+ G_M$ with respect to $m$), hence, by the definition of $w_M$ and $\beta_M$,
\begin{multline*}
\int_Q(-F+G_M)^*(x,t, \beta_M) + (-F+ G_M)(x, t, m_M) +  m_M L\left(-\frac{ w_M}{ m_M}\right) dxdt \\ + \int_{\T} u_T m_M(T) -  u_M(0)m_0 dx = 0.
\end{multline*}
Following Step 3 of the proof of Proposition 5.4  in \cite{CGPT} we can prove that 
$$
\limsup_ {M\to +\infty}\int_{\T}  u_M(0)m_0 dx \leq \int_{\T}  \tilde u(0)m_0.
$$
Hence passing to the limit, due to the convexity of the functionals involved, we have
\begin{equation*}
\int_Q(-F+G)^*(x,t, \tilde \beta) + (-F+ G)(x, t, \tilde m) + \tilde m L\left(-\frac{ \tilde w}{\tilde m}\right) dxdt + \int_{\T} u_T \tilde m(T) -  \tilde u(0)m_0 dx \leq 0.
\end{equation*}

Using the  convexity of $-F+G$,
\begin{equation}\label{ineq1}
(-F+G)^*(x,t, \tilde \beta) + (-F+ G)(x, t, \tilde m)-\tilde \beta(x,t)\tilde m(x,t)\geq 0.
\end{equation}
Therefore \eqref{ineq} is an equality, hence by Lemma 5.3 in \cite{CGPT}, we have $\tilde w=- \tilde m\nabla H(\nabla \tilde u) $. Moreover \eqref{ineq1} holds a.e., hence $\tilde \beta= (-f + g)(\cdot,\cdot, \tilde m)$. This proves that $(\tilde u, \tilde m)$ is a weak solution for  \eqref{MFGbar}. 
Note that this also implies that
$$
\int_Q(-F+G)^*(x,t, \tilde \beta) + (-F+ G)(x, t, \tilde m) + \tilde m L\left(-\frac{ \tilde w}{\tilde m}\right) dxdt + \int_{\T} u_T \tilde m(T) -  \tilde u(0)m_0 dx = 0.
$$

{\bf 2.} Suppose now that $\gamma' > N+2$, so that, by Lemma \ref{min_seq_lem}, $\bar m \in C^{0,\theta}(Q)$, for a $\theta>0$. Then, Theorem 6.4 in \cite{CGPT} applies directly giving   a weak solution $(\tilde u, \tilde m)$ to our convex MFG problem \eqref{MFGbar}, with no need to truncate $\bar{m}$, since there exists $M$ such that $0\leq \bar m(x,t)\leq M$ for all $(x,t)\in Q$. 

We just have to show that the weak solution $(\tilde u, \tilde m)$ enjoys more regularity and it is indeed a  classical solution of \eqref{MFGbar}. Indeed, by Corollary \ref{comp}, also $\tilde m \in C^{0,\theta}(Q)$, hence $\tilde \beta:=  (-f + g)(\cdot,\cdot, \tilde m)$ is H\"older continuous on $Q$; moreover, we know by \cite{CGPT} that the pair $(\tilde u, \bar \beta)$ is a minimiser of
\begin{equation}\label{dualpb}
\inf_{(u,\beta) \in \overline{\mathcal K}} \int_Q (-F+G)^*(x, \beta(x,t)) \, dx dt - \int_\T u(0,x) m_0(x) \, dx, 
\end{equation}
where $\overline{\mathcal K}$ is the set of pairs $(u,\beta)$ satisfying $(u,\beta)\in L^q(Q)\times L^{(\alpha+1)'}(Q)$, (where $(\alpha +1 )' $ is the conjugate exponent  of $\alpha +1 $ and $q$ is as in Definition \ref{def:weaksolMFG}), and 
$$
 -  u_t-\Delta  u +H(\nabla u)\leq  \beta \quad {\rm in }\; Q, 
$$
and $u(T, \cdot) \le u_T(\cdot)$ in the sense of distributions. We may then consider the classical solution $u_1$ to 
$$
 -  (u_1)_t-\Delta  u_1 +H(\nabla u_1) = \tilde \beta(x,t) \quad {\rm in }\; Q, 
$$
$u(T, \cdot) = u_T(\cdot)$. Then, by comparison, the couple $(\tilde u_1, \tilde \beta)$ is still a minimiser of \eqref{dualpb}, so $( u_1, \tilde{m})$ is also a solution to \eqref{MFGbar} (again by Theorem 6.4 in \cite{CGPT}). Since $u_1 \in C^2(Q)$, $\tilde{m}$ is a classical solution of the Fokker-Planck equation by elliptic regularity.

\end{proof}

Since the existence of a solution of the convexified problem has been established, we are now ready to conclude the proofs of Theorems \ref{ex_clas} and \ref{ex_weak}.

\begin{proof}[Proof of Theorem \ref{ex_clas} and Theorem \ref{ex_weak}]

We just need to show that the solution $(\tilde{u}, \tilde m)$ of \eqref{MFGbar} given by Proposition \ref{mfgbarex} is such that $\tilde{m} = \bar{m}$ and $\tilde{w}=\bar {w}$, where $(\bar{m},\bar{w})$ is given by Lemma \ref{min_seq_lem}. This will immediately imply that $(\tilde{u}, \tilde{m})=(\tilde{u}, \bar{m})$ is not only a solution of \eqref{MFGbar}, but also a solution of \eqref{MFG}, and the couple  $(\bar m, \bar m\nabla H(\nabla \tilde u))$ is a minimiser of ${\Ecal}$ in $\Kcal$. 

We observe that $\overline{\Ecal}(m,w) \ge {\Ecal}(m,w)$ (since $G \ge 0$), so

 \[
 \overline{\Ecal}(\tilde m ,\tilde w)=\inf_{(m,w) \in \Kcal} \overline{\Ecal}(m ,w) \geq   \inf_{(m,w) \in \Kcal} \Ecal (m,w) = \Ecal (\bar m,\bar w)= \overline{\Ecal} (\bar m,\bar w). \]
Hence $ (\bar m,\bar w)$ is also a minimiser of $ \overline{\Ecal}$. By strict convexity of $ \overline{\Ecal}$, $ \bar m=\tilde m$ and $\frac{\bar w}{\bar m}= \frac{\tilde w}{\tilde m} $ on the set where $\tilde m=\bar m=0$. Being $\bar w= 0 $ when $\bar m =0$, we also have $\bar w=\tilde w$, as wanted.

\end{proof}

\subsection{$2 < \gamma' \le N+2$: smooth solutions}

We show in this section that, when $2 < \gamma' \le N+2$,  we can find smooth solutions through a penalisation argument under additional hypothesis on $\alpha$.

We consider the approximated (or penalised) Lagrangian
\[
L_\eta (q) := L(q) + \frac{\eta}{N+3}|q|^{N+3}, \quad \forall q \in \Rset^N, \eta > 0,
\]
and the associated functional (defined on $\Kcal$ as before)
\[
\Ecal_\eta(m, w) = \int_Q m L_\eta\left(-\frac{w}{m}\right) - F(x,m) \, dx dt + \int_\T u_T(x) m(x,T) \, dx.
\]
We are basically increasing the growth of $L$ in order to gain regularity for minimisers of the energy. In particular, if $L$ grows faster that $|q|^{N+2}$, a solution of the Fokker-Planck equation, that enters in the constraint $\Kcal$, enjoys automatically H\"older regularity (see Corollary \ref{comp}). Note that a similar penalisation argument has been implemented in \cite{MeSil} in the stationary setting.
For any fixed $\eta > 0$, $L_\eta(q)$ behaves like $|q|^{N+3}$ as $|q| \to +\infty$, namely
\[
c_\eta^{-1}|q|^{N+3} - c_\eta \le L_\eta(q) \le c_\eta(|q|^{N+3} + 1),
\]
for all $q \in \RsetN$ and some positive $c_\eta$ (depending on $\eta$). The corresponding family of Hamiltonians $H_\eta$ satisfy 
\[
h_\eta^{-1}|p|^{\frac{N+3}{N+2}} - h_\eta \le H_\eta(p) \le h_\eta(|p|^{\frac{N+3}{N+2}} + 1),
\]
together with additional bounds independent of $\eta$. In particular, we have by \eqref{Lass} that
\begin{equation}\label{Hetasubgamma}
-C_L\leq -L(0) \le H_\eta(p) = \sup_{q \in \Rset^N} [p \cdot q - L_\eta(q)] \le \sup_{q \in \Rset^N} [p \cdot q - L(q)] = H(p) \le C_H (|p|^{\gamma} + 1),
\end{equation}
where $C$ does not depend on $\eta$. Moreover,
\begin{equation}\label{Heta_p2}
|\nabla H_\eta(p)| \le C(|p|^{\gamma-1} + 1).
\end{equation}
Indeed, by the definition of the Legendre transform,
\[
L_\eta(\nabla H_\eta(p)) = \nabla H_\eta(p) \cdot p - H_\eta(p) \quad \forall p \in \mathbb R^N.
\]
Therefore, by \eqref{Hetasubgamma},
\[
C_L^{-1} | \nabla H_\eta(p)) |^{\gamma'} - C_L \le L_\eta(\nabla H_\eta(p)) \le | \nabla H_\eta(p)| \cdot |p| + C_L\le\frac{C_L^{-1}}{2} | \nabla H_\eta(p)) |^{\gamma'} + C|p|^\gamma + C_L ,
\]
which implies \eqref{Heta_p2}, as $\gamma/\gamma'=\gamma-1$.

When $\eta > 0$ is fixed, $H_\eta$ satisfies \eqref{Hass} with $\gamma=\frac {N+3}{N+2}$, hence its conjugate $\gamma' = N+3 > N+2$, so Theorem \ref{ex_clas} applies. In particular, there exists a classical solution $(u_\eta, m_\eta)$ of
\begin{equation}\label{MFGeta}
\begin{cases}
-u_t - \Delta u + H_\eta(\nabla u) = -f(x,m(x,t)), \\
m_t - \Delta m - \diverg(\nabla H_\eta(\nabla u)\, m)  = 0 & \text{in $Q$,} \\
m(x,0) = m_0(x), \quad u(x,T) = u_T(x) & \text{on $\T$}.
\end{cases}
\end{equation}
such that, setting $w_\eta:= -m_\eta\nabla H(\nabla u_\eta)$, then $(m_\eta , w_\eta)$ is a minimiser of $\Ecal_\eta$. We will show that $(u_\eta, m_\eta)$ converges as $\eta \to 0$ to a solution of the original problem. Before proving Theorem \ref{ex_clas_2} we state some a-priori estimates that will be crucial to pass to the limit.

\begin{lem}\label{lem_p} For all $p \in [1, \frac{N+2}{N+2-\gamma'})$, there exists $C_p > 0$ such that
\begin{equation}\label{N2norm}
\|m_\eta\|_{L^p(Q)} \le C_p.
\end{equation}
\end{lem}

\begin{proof} We observe that $\Ecal_\eta(m,w) \le \Ecal_1(m,w)$ for all $\eta \le 1$ and $(m,w) \in \Kcal$. Hence, $\Ecal_\eta(m_\eta,w_\eta) = \min \Ecal_\eta \le \min \Ecal_1$ for all $\eta \le 1$. Since
\[
L_\eta (q) = L(q) + \frac{\eta}{N+3}|q|^{N+3} \ge  C_L^{-1}|q|^{\gamma'} - C_L,
\]
arguing as in Lemma \ref{Eboundbelow} and recalling that $m_\eta$ solves the Fokker-Planck equation with drift $A_\eta = \nabla H_\eta(\nabla u_\eta)$, we get
\[
\int_Q m_\eta^{\alpha+1}\, dx dt + \int_Q |\nabla H_\eta(\nabla u_\eta))|^{\gamma'} m_\eta \, dxdt \le C.
\]
Then we can apply Corollary \ref{comp} to conclude.
\end{proof}

\begin{lem}\label{lem_infty} Suppose that  
 \[
\alpha < \min\left \{ \frac{\gamma'}{N}, \frac{\gamma'-2}{N+2-\gamma'}\right\}.
\]
Then, there exists $C > 0$ such that
\begin{equation}\label{inftybound}
\|m_\eta\|_{L^\infty(Q)} \le C
\end{equation}
for all $\eta > 0$.
\end{lem}

\begin{proof} By contradiction, let $M_\eta > 0, x_\eta \in \T$ be such that
\[
0 < M_\eta := m_\eta(x_\eta, t_\eta) = \max_{(x,t)\in \overline{Q}} m_\eta (x,t) \to \infty, \quad \text{as $\eta \to 0$}.
\]
{\bf 1.} Let us define the following blow-up sequences 
\begin{equation}\label{blowupdef}
v_\eta(x,t) := a_\eta^{\gamma'-2} u_\eta(x_\eta + a_\eta x, t_\eta + a^2_\eta t), \quad \mu_\eta(x, t) := \frac{1}{M_\eta} m_\eta(x_\eta + a_\eta x, t_\eta + a^2_\eta t), \quad a_\eta = M_\eta^{-\alpha/(\gamma'-2)},
\end{equation}
for all $(x, t) \in Q_\eta := \{(x, t) \in \Rset^{N+1} : (x_\eta + a_\eta x, t_\eta + a^2_\eta t) \in Q\}$. Then, $(v_\eta, \mu_\eta)$ solves
\begin{equation}\label{MFGn}
\begin{cases}
- (v_\eta)_t - \Delta v_\eta + \widehat{H}_\eta(\nabla v_\eta)= - f_\eta(x, \mu_\eta)  \\
( \mu_\eta)_t - \Delta \mu_\eta(x) -{\rm div}(\nabla \widehat{H}_\eta(\nabla v_\eta(x)) \, \mu_\eta(x)) = 0  & \text{in $Q_\eta$}, \\
\mu_\eta \left(x,-t_\eta/a_\eta^2 \right) = \frac{1}{M_\eta} m_0(x_\eta + a_\eta x), \\ v_\eta \left(x,(T -t_\eta)/a_\eta^2 \right)  = a_\eta^{\gamma'-2}u_T(x_\eta + a_\eta x) & \text{on $T_\eta$ },
\end{cases}
\end{equation}
where $T_\eta = \{x: x_\eta + a_\eta x \in \T\}$, $\widehat{H}_\eta(p) = a_\eta^{\gamma'}H_\eta(a_\eta^{1-\gamma'}p)$, $f_\eta(x, \mu) = a_\eta^{\gamma'}f(x_\eta + a_\eta x, M_\eta \mu_\eta)$. Note that
 \[a_\eta \to 0 \quad \text{as }  \eta\to0,\] 
and $\widehat H_\eta$ satisfies for some $C_1 > 0$
\begin{equation}\label{Hhatbound}
- a_\eta^{\gamma'} C_1 \le \widehat{H}_\eta(p) \le C_H(|p|^\gamma + a_\eta^{\gamma'}) \le C_1(|p|^2 + 1)
\end{equation}
for all $\eta>0$ and $p\in\RsetN$, by \eqref{Hetasubgamma}. Similarly,
\begin{equation}\label{Hhatbound2}
|\nabla \widehat{H}_\eta(p)| \le C(|p|^{\gamma-1} + 1),
\end{equation}
by \eqref{Heta_p2}.

Moreover, $\mu_{\eta} \le 1$ on $Q_\eta$ and 
\begin{equation}\label{Fbound}
0 \le f_\eta(x, \mu_\eta) \le C_f a_\eta^{\gamma'} (M_\eta^\alpha + 1) \le C_f(a_\eta^2 + a^{\gamma'}_\eta)
\end{equation}
for all $\eta$, by \eqref{fass}. Moreover, since $\gamma' \ge 2$, $\mu_\eta$ and $v_\eta$ are bounded (uniformly in $\eta$) in $W^{2,\infty}(T_\eta)$ at initial and final time respectively.

{\bf 2.} We show that $v_\eta$ and its gradient are bounded on $Q_\eta$. It suffices to observe that
\[
\begin{split}
& \bar{v}_\eta(x,t):= a_\eta^{\gamma'-2}\sup_{x\in\T} u_T(x) + C_1 a_\eta^{\gamma'}\left(\frac{T-t_\eta}{a_\eta^2} - t \right), \\
& \underline{v}_\eta(x,t) := a_\eta^{\gamma'-2}\inf_{x\in\T} u_T (x) - (C_H a_\eta^{\gamma'} + C_f a_\eta^{\gamma'} + C_f a_\eta^2)\left(\frac{T-t_\eta}{a_\eta^2} - t \right)
\end{split}
\]
are respectively supersolutions and subsolutions of the (backward) Cauchy problem for the HJB equation in \eqref{MFGn}. Hence,
\begin{equation}\label{vbound}
-C \le \underline{v}_\eta(x,t) \le v_\eta(x,t) \le \bar{v}_\eta(x,t) \le C \quad \text{on $Q_\eta$}
\end{equation}
by the Comparison Principle. Proposition \ref{hjb_regularity} applies as \eqref{Hhatbound}, \eqref{Fbound}, \eqref{vbound} hold, thus $$\|\nabla v_\eta \|_{L^\infty(Q_\eta)} \le C.$$

{\bf 3.} The rescaled distribution $\mu_\eta$ is a solution of the following linear equation
\[
(\mu_\eta)_t - \Delta \mu_\eta  = {\rm div}(\Phi_\eta(x,t)) \quad \text{in $Q_\eta$},
\]
where $\Phi_\eta(x,t) = \nabla \widehat{H}_\eta(\nabla v_\eta(x,t))  \mu_\eta(x,t)$ is bounded in $L^\infty(Q_\eta)$ uniformly with respect to $\eta$, by the previous step and \eqref{Hhatbound2}. Note that $|\mu_\eta|$ itself is bounded by one. Thus, by classical elliptic regularity (see, for example, \cite[Theorem V.1.1]{lady}), we conclude that $\|\mu_\eta\|_{C^{0,\theta}(Q_\eta)} \le C$ for some $\theta > 0$. It follows that $\mu_\eta$ is bounded away from zero in a neighbourhood of zero, as $\mu_\eta(0,0) = 1$. Therefore, for any fixed $p > 0$, there exists some neighbourhood $U$ of $(x,t)=(0,0)$ and $\delta > 0$, depending on $p$ but not on $\eta$, such that
\begin{equation}\label{belowLalpha}
\int_{U} (\mu_\eta)^p \, dx dt \ge \delta.
\end{equation}
We may choose $p$ so that
\[
\alpha \frac{N+2}{\gamma'-2} < p < \frac{N+2}{N+2-\gamma'},
\]
by the assumptions on $\alpha$. Thus, in view of \eqref{N2norm},
\[
\int_{Q_\eta} (\mu_\eta)^p \, dx dt = \frac{1}{M_\eta^p} \int_{Q_\eta} (m_\eta(x_\eta + a_\eta x, t_\eta + a^2_\eta t))^p \, dx  dt \\ = M^{\alpha \frac{(N+2)}{\gamma'-2} - p} \, \|m_\eta \|_{L^p(Q)}^p \to 0
\]
as $\eta \to 0$, by the assumptions on $\alpha$ and the fact that $M_\eta \rightarrow \infty$, but this contradicts \eqref{belowLalpha}.

\end{proof}

\begin{proof}[Proof of Theorem \ref{ex_clas_2}] Once $L^\infty(Q)$ estimates on $m_\eta$ are in force, we just have to improve the bounds on $(u_\eta, m_\eta)$ in order to pass to (classical) limits in \eqref{MFGeta}. By Proposition \ref{hjb_regularity}, \eqref{Hetasubgamma} and \eqref{inftybound}, $\|u_\eta\|_{C^{1,\theta}(Q)}$ is bounded independently on $\eta$, for some $\theta>0$. Note that such a bound can be extended to $(0,T]$ by the regularity of the final datum $u_T$. Standard elliptic regularity then provides H\"older bounds on the second derivatives of $u_\eta$. The same assertion holds for $m_\eta$, by elliptic regularity applied to the Fokker-Planck equation. Finally, note that also $-w_\eta/m_\eta = \nabla H(\nabla u_\eta)$ is bounded in $L^\infty(Q)$ independently on $\eta$. Therefore the penalisation term $\frac{\eta}{N+3}\int_Q m_\eta |w_\eta/m_\eta|^{N+3}$ in $\Ecal_\eta$ vanishes, implying
\[
\min_{\Kcal} \Ecal_\eta \to\min_{\Kcal} \Ecal \quad \text{as $\eta \to 0$.}
\]

\end{proof}

\begin{rem} In Lemma \ref{lem_infty}, the rescaling is designed so that $(v_\eta, \mu_\eta)$ solves the Hamilton-Jacobi-Bellman equation
\[
- ( v_\eta)_t - \Delta v_\eta + \widehat{H}_\eta(\nabla v_\eta)= - f_\eta(x, \mu_\eta),
\]
where $|f_\eta| \le 1$ on $Q$ for all $\eta$. Then, gradient bounds $\|\nabla v_\eta \|_{L^\infty(Q_\eta)} \le C$ are used to run an argument by contradiction; those bounds are obtained by first estimating $\|v_\eta \|_{L^\infty(Q_\eta)} \le C$ through a comparison principle. Lipschitz estimates then follow.

Note that gradient estimates available in the literature usually depend on bounds on the solution itself. On the other hand, for {\it stationary} HJB equations, namely
\[
 \lambda - \Delta v(x) + H(\nabla v(x)) = F(x),
\]
it is possible to prove Lipschitz estimates that {\it do not} depend a priori on $\|v\|_{L^\infty(Q)}$, for example by means of the Bernstein method (see \cite{LL89}). This key fact has been used in \cite{cir16} to prove existence of classical solutions to \eqref{MFG} in the stationary case, for a wider range of couplings $f$. If similar estimates were available also in the time dependent case, i.e. $\|\nabla v_\eta \|_{L^\infty(Q_\eta)} \le C$, with $C$ depending on $\|f_\eta\|_{L^\infty(Q_\eta)}$, $\|v_\eta(\cdot, T)\|_{L^\infty(Q_\eta)}$ but not on $\|v\|_{L^\infty(Q)}$ and the time horizon $T$, then it would have been possible to have Lemma \ref{lem_infty} in the whole range $\alpha \in (0, \gamma'/N)$, and hence the existence of classical solutions. Note that our $v_\eta$ and $f_\eta$ are space-periodic, but with no control on the period.
\end{rem}

\section{A non-uniqueness example} \label{non-uniqueness}

The aim of this section is to prove that, under additional assumptions, \eqref{MFG} has multiple solutions. Consider $f(x,m) = f(m)$ and $u_T \equiv 0$. Note that the corresponding system
\begin{equation}\label{MFGnonu}
\begin{cases}
-u_t - \Delta u + H(\nabla u) = -f(m(x,t)) & \text{in $Q$,}\\
m_t - \Delta m - \diverg(\nabla H(\nabla u)\, m)  = 0 & \text{in $Q$,} \\
m(x,0) = 1, \quad u(x,T) = 0 & \text{on $\T$}.
\end{cases}
\end{equation}
has always a trivial solution $(\bar{u},\bar{m}) = ((t-T)f(1),1)$. We look for a solution to \eqref{MFGnonu} which is not the trivial one.

We will require that for some $b >2$, $c > 0$,
\begin{equation}\label{nonucond}
\begin{split}
\bullet \quad & 0\leq L(q) \le c |q|^b \quad \text{for all $|q| \le 1$,}\\
\bullet \quad & f'(m) \ge c \quad \text{for all $|m| \le 2$}.
\end{split}
\end{equation}

\begin{prop}\label{nonunique} Under the assumptions of Theorem \ref{ex_clas} and \eqref{nonucond}, there exist at least two (different) classical solutions to \eqref{MFGnonu} if $T$ is large enough.
\end{prop}

\begin{proof}
In view of Theorem \ref{ex_clas}, there exists a classical solution $(\tilde{u}, \tilde{m})$ of \eqref{MFG} such that the couple $(\tilde{m}, \tilde{w})=(\tilde{m} , -\tilde{m}\nabla H(\nabla \tilde{u}))$ is a minimiser of $\Ecal$. Our aim is to show that $(\tilde{u}, \tilde{m})$ cannot be the trivial solution; this will be achieved by showing that $\Ecal(\tilde{m} , -\tilde{m}\nabla H(\nabla \tilde{u})) = \min_\Kcal \Ecal(m, w)< \Ecal(1, 0) = - T F(1)$.

In order to build a suitable competitor $(m,w)$, let us consider $\mu(x) := 1 + \epsilon \varphi(x)$, where $\epsilon>0$, $\varphi \in C^2(\T)$ is fixed and satisfies $\int_\T \varphi dx = 0$. Set $v(x) := \nabla \mu(x) = \epsilon \nabla  \varphi(x)$. Recalling that $F''(m) = f'(m)$, by \eqref{nonucond} we have
\[
\int_\T \mu L\left(-\frac{v}{\mu}\right) - F(\mu)\, dx \le -F(1) + c \epsilon^b \int_\T (1+ \epsilon \varphi)^{1-b}|\nabla \varphi|^b\, dx - c \epsilon^2  \int_\T \varphi^2 \, dx
\]
if $\epsilon > 0$ is small. By reducing $\epsilon > 0$, there exists $\delta > 0$ so that
\begin{equation}\label{nonu1}
\int_\T \mu L\left(-\frac{v}{\mu}\right) - F(\mu)\, dx \le -F(1) - \delta.
\end{equation}

The couple $(\mu, v)$ will serve as a competitor in the ``long-time'' regime, that is in the time interval $[1, T]$. In order to connect the initial datum $m_0 \equiv 1$ to $\mu$ at $t=1$, we define
\[
m(x,t):=1 + \zeta(t) (\mu(x)-1), \quad w(x,t) = \zeta(t) v(x) - \zeta'(t) \nabla \psi(x) \quad \text{on $Q$},
\]
where $\zeta : [0,T] \to \Rset$ is a smooth function which is zero at $t=0$, equal to one in the interval $[1, T]$, and $\psi \in C^2 (\T)$ is a solution of $\Delta \psi = \mu - 1$ on $\T$. One easily verifies that $(u, m ) \in \Kcal$. Moreover,
\begin{align}
\Ecal(m, w) = &\int_0^1 \int_\T (1 + \zeta(t) (\mu(x)-1)) L\left(-\frac{\zeta(t) v(x) - \zeta'(t) \nabla \psi(x)}{1 + \zeta(t) (\mu(x)-1)}\right) - F(1 + \zeta(t) (\mu(x)-1)) \, dx dt \nonumber\\ &+ \int_1^T \int_\T \mu L\left(-\frac{v}{\mu}\right) - F(\mu) \, dx dt.
\end{align}
The first integral in the previous equation is finite (as the integrand is in $L^\infty(\T)$), and can be bounded by a positive constant $C$ not depending on $T$. As for the second integral, we use \eqref{nonu1} to obtain
\[
\Ecal(m, w) \le C - (T-1)(F(1) + \delta) < - T F(1) = \Ecal(1, 0)
\]
if $T$ is large enough. Hence, the minimum of $\Ecal$ is not achieved by $(1, 0)$, and $(\tilde{u}, \tilde{m})$ cannot be the trivial solution.

\end{proof}

\begin{rem} The conclusion of Proposition \ref{nonunique} holds if one replaces the assumptions of Theorem \ref{ex_clas} with the assumptions of Theorem \ref{ex_clas_2} or Theorem \ref{ex_weak}. In the latter case, the non-trivial solution minimising the energy $\Ecal$ has to be intended in the weak sense.

Condition \eqref{nonucond} can be also weakened: the construction of a minimum of  $\Ecal$ that cannot be the trivial solution basically relies on the existence of a couple $(v, \mu)$ satisfying the constraint
\[
\Delta \mu + {\rm div}(v) = 0 \text{ on $\T$,} \quad \int_\T \mu \,dx = 1, \quad \mu \ge 0,
\]
such that
\[
{\Ecal}_S(\mu, v) := \int_\T \mu L\left(-\frac{v}{\mu}\right) - F(\mu)\, dx < L(0)-F(1).
\]
\end{rem}
In other words, this can be seen as requiring that the energy functional ${\Ecal}_S$ associated to the {\it stationary} version of \eqref{MFGnonu} is not minimised by the couple $(1,0)$. 

\appendix

\section{Existence for small $T$} \label{appendix}

To prove existence in the small time-horizon regime, we implement a standard contraction mapping principle (see, for example, \cite[Chapter 15]{taylor}). This tool has already been used in \cite{Ambrose} in the MFG framework to prove existence of solutions for \eqref{MFG}, but with a different spirit: existence for arbitrarily large $T$ but initial data close to $\bar{m} \equiv 1$ (in the mentioned work, the functional space setting is indeed different). 

Let us rewrite \eqref{MFG} in integral form; set $v(\cdot, t) := u(\cdot, T-t)$ for all $t \in [0,T]$, then
\begin{equation}\label{FFI}
\begin{cases}
v(x,t) = e^{t \Delta} u_T(x) - \int_0^t e^{(t-s)\Delta} \Phi^v[v,m](s)(x) ds, \\
m(x,t) = e^{t \Delta} m_0(x) + \int_0^t e^{(t-s)\Delta} \Phi^m[v,m](s)(x) ds,
\end{cases}
\end{equation}
where
\[
\begin{split}
& \Phi^v[v,m](s)(\cdot) := f(m(\cdot, T-s)) - H(\nabla v(\cdot, s)), \\ & \Phi^m[v,m](s)(\cdot): = {\rm div}(\nabla H (\nabla v (\cdot, T-s)) m(s)) \quad \forall s \in [0,T]
\end{split}
\]
and $ e^{t \Delta}$ is the (strongly continuous) semigroup associated with the parabolic equation $\varphi_t = \Delta \varphi$ and defined on suitable H\"older spaces (see {\it iii)} below). Note that \eqref{FFI} has the form of a forward-forward system for $v,m$. Here, the {\it local} regularity of $H,f$ plays a role, rather than the time direction in the two equations or the behaviour of $f$ at infinity: we just assume $f, H \in C^3$ and do not require \eqref{Hass} and \eqref{fass} to hold. We stress that this argument could be adapted to more general MFG systems (congestion problems, $u_T$ depending on $m$, ...).

Let us define $X^{k,\nu} := C([0,T], C^{k,\nu}(\T))$ for all integers $k$ and $\nu \in (0,1)$. We will need the following facts: let $0 < \nu < \beta < 1$, then
\begin{itemize}
\item[{\it i)}] If $h \in C^2(\RsetN)$ and $\|g_1\|_{C^{1,\nu}(\T)}, \|g_2\|_{C^{1,\nu}(\T)} \le K$, then $\|h(g_1)-h(g_2)\|_{C^{1,\nu}(\T)} \le C\|g_1-g_2\|_{C^{1,\nu}(\T)}$, for some $C = C(h,K,\nu) > 0$. 
\item[{\it ii)}] $\|g_1 g_2\|_{C^{\nu}(\T)} \le C \|g_1\|_{C^{\nu}(\T)} \| g_2\|_{C^{\beta}(\T)}$ for all $g_1 \in C^{\nu}(\T)$,  $g_2 \in C^{\beta}(\T)$.
\item[{\it iii)}] $\| e^{t \Delta} u \|_{C^{2,\nu}(\T)} \le C t^{-\frac{2+\nu-\beta}{2}}\|u\|_{C^{0,\beta}(\T)}$, $\| e^{t \Delta} u \|_{C^{2,\beta}(\T)} \le C t^{-\frac{1}{2}}\|u\|_{C^{1,\beta}(\T)}$ for all $t \in (0,1]$.
\end{itemize}
Items {\it i)} and {\it ii)} follows by computation; as for {\it iii)} see, for example, \cite{taylor}, p 274. Recall also that $C^{2,\nu}(\T)$ is continuously embedded into $C^{1,\beta}(\T)$.

Note that the contraction mapping principle implies also uniqueness of solutions in the space $\mathcal{Z}_a$ (see \eqref{Za} below), that is, there is only one equilibrium $(u,m)$ close to the final-initial data $(u_0, m_0)$.

\begin{proof}[Proof of Theorem \ref{short_ex}] Fix any $0 < \nu < \beta < 1$. Then, $\Phi^v : X^{2,\beta} \times X^{2,\nu} \to X^{1,\beta}$ and $\Phi^m : X^{2,\beta} \times X^{2,\nu} \to X^{0,\beta}$. Let $ a > 0$ and
\begin{multline}\label{Za}
\mathcal{Z}_a := \{ (v,m) \in X^{2,\beta} \times X^{2,\nu} : v(0) = u_T, \, m(0) = m_0, \\ \|v(t) - u_T\|_{C^{2,\beta}(\T)} \le a, \, \|m(t) - m_0\|_{C^{2,\nu}(\T)} \le a \text{ for all $t \in [0,T]$}\}.
\end{multline}
Let $(v,m) \mapsto (\hat{v}, \hat{m}) := \Psi(v,m)$, where
\[
\begin{cases}
 \hat{v}(t) = e^{t \Delta} u_T - \int_0^t e^{(t-s)\Delta} \Phi^v[v,m](s) ds, \\
 \hat{m}(t) = e^{t \Delta} m_0 + \int_0^t e^{(t-s)\Delta} \Phi^m[v,m](s) ds.
\end{cases}
\]
Our aim is to prove that $\Psi$ has a fixed point, by means of the contraction mapping theorem. First, we claim that $\Psi$ maps $\mathcal{Z}_a$ into itself when $T = T(a)$ is small. Indeed,
\[
\|e^{t \Delta} u_T - u_T \|_{C^{2,\beta}(\T)} \le a/2, \quad \|e^{t \Delta} m_0 - m_0\|_{C^{2,\nu}(\T)} \le a/2
\]
if $t$ is small, by continuity of the semigroup $e^{t \Delta}$. Moreover,
\begin{align*}
\left\lVert \int_0^t e^{(t-s)\Delta} \Phi^v[v,m](s) ds \right\rVert_{C^{2,\beta}(\T)} &\le \int_0^t \| e^{(t-s)\Delta} \Phi^v[v,m](s) \|_{C^{2,\beta}(\T)} ds \\ &\le  C \int_0^t (t-s)^{-1/2}\| \Phi^v[v,m](s) \|_{C^{1,\beta}(\T)} ds \le C_1 T^{1/2} \le a/2
\end{align*}
whenever $T$ is small (here, $C$, $C_1$, ..., are positive constants depending on $a$, but not on $T$). Similarly,
\[
\left\lVert \int_0^t e^{(t-s)\Delta} \Phi^m[v,m](s) ds \right\rVert_{C^{2,\nu}(\T)}  \le C T^{(\beta- \nu)/2} \le a/2,
\]
so $\Psi : \mathcal{Z}_a \to \mathcal{Z}_a$.

To show that $\Psi$ is a contraction, note that for all $s \in [0,T]$,
\begin{align*}
&\| \Phi^m[v_1,m_1](s) - \Phi^m[v_2,m_2](s) \|_{C^{0,\beta}(\T)} \le \\& \le \| \nabla H(\nabla v_1 (T-s)) (m_1- m_2)  \|_{C^{1,\beta}(\T)} + \| m_2 \,[ \nabla H(\nabla v_1 (T-s)) -  \nabla H(\nabla v_2 (T-s))]  \|_{C^{1,\beta}(\T)}  \\
&\le C  \| \nabla H(\nabla v_1 (T-s))\|_{C^{1,\beta}(\T)} \|m_1 - m_2  \|_{C^{1,\beta}(\T)} + \| m_2 \|_{C^{1,\beta}(\T)} \| (\nabla v_1 -  \nabla v_2) (T-s) \|_{C^{1,\beta}(\T)} \\ 
& \le C_1( \|m_1 - m_2  \|_{C^{2,\nu}(\T)} + \| v_1 (T-s)-  v_2 (T-s) \|_{C^{2,\beta}(\T)}).
\end{align*}
Therefore,
\begin{align*}
&\left\lVert \int_0^t e^{(t-s)\Delta}\left( \Phi^m[v_1,m_1](s) - \Phi^m[v_2, m_2](s)\right)ds \right\rVert_{C^{2,\nu}(\T)} \le\\ &\le C \int_0^t (t-s)^{-\frac{2+\nu-\beta}{2}} ( \|m_1 - m_2  \|_{C^{2,\nu}(\T)} + \| v_1 (T-s)-  v_2 (T-s) \|_{C^{2,\beta}(\T)}) ds
 \\ &\le C T^{(\beta- \nu)/2} \sup_{s \in [0,T]}( \|m_1 - m_2  \|_{C^{2,\nu}(\T)} + \| v_1 (T-s)-  v_2 (T-s) \|_{C^{2,\beta}(\T)})\\  &\le  \|m_1 - m_2  \|_{X^{2,\nu}} + \| v_1 -  v_2 \|_{X^{2,\beta}} .
\end {align*}
by eventually reducing $T$. In a similar way, one shows that
\[
\left\rVert\int_0^t e^{(t-s)\Delta} \Phi^v[v_1,m_1](s) - \Phi^v[v_2, m_2](s)ds \right\rVert_{C^{2,\beta}(\T)} \le \|m_1 - m_2  \|_{X^{2,\nu}} + \| v_1 -  v_2 \|_{X^{2,\beta}},
\]
hence $\Psi$ is a contraction; its fixed point in $\mathcal{Z}_a$ is a solution to \eqref{FFI}, and hence a classical solution to \eqref{MFG} by classical Schauder regularity results.

\end{proof}

\medskip
\begin{flushright}
\noindent \verb"cirant@math.unipd.it"\\
Dipartimento di Matematica ``Tullio Levi-Civita''\\ Universit\`a di Padova\\
via Trieste 63, 35121 Padova (Italy)

\smallskip

\noindent \verb"tonon@ceremade.dauphine.fr"\\
 Universit\'e Paris-Dauphine\\
 PSL Research University\\
 CNRS, Ceremade\\
Place du Mar\'echal de Lattre de Tassigny\\
75775 Paris cedex 16 (France)
\end{flushright}

\end{document}